\documentclass[11pt]{article}
\usepackage{amsmath, amssymb, amsthm, mathrsfs, amscd, xypic}
\usepackage[dvips]{graphics} 
\usepackage{xypic}
\xyoption{all}
\newcommand{\ds}{\displaystyle}

\newcommand{\QQ}{\mathbb{Q}}
\newcommand{\RR}{\mathbb{R}}
\newcommand{\CC}{\mathbb{C}}
\newcommand{\NN}{\mathbb{N}} 
  
\newcommand{\KK}{\mathbb{K}}
\newcommand{\HH}{\mathbb{H}} 
\newcommand{\DD}{\mathbb{D}}
\newcommand{\EE}{\mathbb{E}}
\newcommand{\Aff}{\mathbb{A}}

\newcommand{\Aut}{\operatorname{Aut}}  
\newcommand{\End}{\operatorname{End}}
\newcommand{\id}{\operatorname{id}} 
   
\newcommand{\Sp}{\operatorname{span}}

\newcommand{\Hom}{\operatorname{Hom}}

\newcommand{\e}{\epsilon}

\newcommand{\Endo}{\operatorname{End}} 
\newcommand{\op}{\operatorname{op}}

\newcommand{\Res}{\operatorname{Res}}

\newtheorem{theorem}{Theorem}[section]
\newtheorem{proposition}[theorem]{Proposition}
\newtheorem{lemma}[theorem]{Lemma}
\newtheorem{corollary}[theorem]{Corollary}
\newtheorem{definition}[theorem]{Definition}

\newtheorem{example}[theorem]{Example}

\theoremstyle{remark}
\newtheorem{remark}[theorem]{Remark}

\renewenvironment{proof}{{\noindent\bf Proof.}}{\hfill $\Box$\par\vskip3mm}

\title{Maximal Subalgebras of Finite-Dimensional Algebras}
\author{\sc Miodrag Cristian Iovanov \\   
\vspace{.5cm} {\small  University of Iowa, USA and University of Bucharest, Romania }\\
\sc Alexander Sistko \\
{\small  University of Iowa, USA}
}

\begin{document} 

\maketitle 
\date{}     

\begin{abstract}
\noindent 
We study maximal associative subalgebras of an arbitrary finite dimensional associative algebra $B$ over a field $\KK$, and obtain full classification/description results of such algebras. This is done by first obtaining a complete classification in the semisimple case, and then lifting to non-semisimple algebras. The results are sharpest in the case of algebraically closed fields, and take special forms for algebras presented by quivers with relations. We also relate representation theoretic properties of the algebra and its maximal and other subalgebras, and provide a series of embeddings between quivers, incidence algebras and other structures which relate indecomposable representations of algebras and some subalgebras via induction/restriction functors. Some results in literature are also re-derived as a particular case, and other applications are given.
\footnote{{2000 \textit{Mathematics Subject Classification}. Primary 16S99; Secondary 16G60, 16G10, 16S50, 16W20}\\
{\bf Keywords} maximal subalgebra, semisimple algebra, separable functor, separable, split, split-by-nilpotent}
\end{abstract} 

\section*{Introduction} 

Given a mathematical object, it is often natural to consider its maximal subobjects as a means of further understanding it. Maximal subalgebras of (not-necessarily associative) algebras, and in particular maximal abelian subalgebras, have classically guided such inquiry. A well-known instance of this principle arises, of course, in the structure theory of finite-dimensional semisimple Lie algebras, where a central role is played by their Cartan subalgebras: over $\CC$, these are simply maximal abelian subalgebras, as seen in the classical papers \cite{Dynk}, \cite{Ma}. Others have subsequently generalized this work, and have applied similar ideas to maximal sub-structures of other possibly non-associative algebraic structures such as Malcev algebras, Jordan algebras, associative superalgebras, or classical groups (\cite{Dynk2}, \cite{Eld1}, \cite{Eld2}, \cite{Eld3}, \cite{MZ}, \cite{Rac0}, \cite{Rac}, \cite{Rac1}). On the associative side, a well known result of Schur \cite{Sch} states that a commutative subalgebra of $M_n(\CC )$ can have dimension at most $\lfloor \frac{n^2}{4} \rfloor + 1$ , and this dimension is attained. A nice short proof of this interesting result was given later by Mirzakhani \cite{Mir}. In another related direction, Motzkin and Taussky proved that the variety of commuting $n$-by-$n$ complex matrices is irreducible in \cite{Motz1}, \cite{Motz2}, and Gerstenhaber \cite{Ger2} noted that this bounds the dimension of any $2$-generated abelian subalgebra of $M_n(\CC )$ by $n$. There has been a lot of interest in studying dimensions of certain subalgebras of matrix algebras, as well as irreducible components of matrix varieties satisfying various properties. Often, these questions are tightly related to representation theory (\cite{Bar}, \cite{Bas1}, \cite{Bas2}, \cite{Bas3}, \cite{Ger}, \cite{Ger2}, \cite{Gur2}, \cite{Gur3}, \cite{Prem}, \cite{Motz1}, \cite{Motz2}, \cite{Schr}).

The results of Schur concerning the maximal dimension of commutative subalgebras were generalized in several directions. One such extension is attributed to Jacobson \cite{Jac}, who generalized Schur's theorem to any field $\KK$. Maximal subfields of algebraically closed fields were studied by Guralnick and Miller in \cite{Gur}. At the other end, Laffey \cite{Laff} gave lower bounds for maximal abelian subalgebras of $M_n(\KK )$. In the general case of not necessarily commutative maximal subalgebras of matrix algebras, the problem was studied by Racine, who obtained a structure theorem for maximal subalgebras of associative central simple algebras \cite{Rac0}, \cite{Rac}. Using a deep theorem of Gerstenhaber from \cite{Ger}, in \cite{Ag} Agore showed that the maximal dimension of a subalgebra in $M_n(\CC)$ is $n^2-n+1$. Recently, on the infinite dimensional side, interest for maximal subalgebras arose as well in commutative algebra \cite{Mau}.  However, there does not seem to be a general classification of maximal subalgebras of finite dimensional associative algebras beyond the case of matrix algebras. 


The main result and first goal of this paper is to provide such a complete classification. More precisely, given a finite-dimensional associative algebra $B$, we wish to answer two questions: 
 
\begin{enumerate} 
\item Can we classify/describe the maximal associative subalgebras $A \subset B$? 
\item Can we determine under what conditions a maximal subalgebra $A\subset B$ shares interesting representation-theoretic data with $A$, and what can be said about such a minimal extension of algebras? 
\end{enumerate}  


For instance, as it turns out, a relevant question for (2) above will be finding conditions under which the extension $A \subset B$ is separable, split, or split-by-nilpotent. 

We first provide a general structure theorem for maximal associative subalgebras of any finite dimensional associative algebra $B$. The study proceeds in two steps: one deals first with the semisimple case, where there are essentially four types of maximal subalgebras, and then this is used to ``lift" modulo the Jacobson radical, to deal with the general case. In the semisimple case, our proofs are a blend of techniques characteristic to finite dimensional simple algebras and representation theoretic arguments. In the non-semisimple case, there are two types of subalgebras: one coming essentially from maximal subalgebras of $B/J(B)$ via pull-back, where $J(B)$ is the Jacobson radical (and these further ramify via the semisimple classification); and another one, characterized by the property that irreducible modules of $A$ and $B$ ``coincide" via restriction. These two types will give rise to examples of separable extensions and of split extensions of algebras, respectively, and both situations can be understood as particular cases of separable functors. The results take particularly nice forms when additional mild hypotheses are imposed, such that the algebra $B/J(B)$ is separable (in particular, when $\KK$ is algebraically closed, or when it is a splitting field for $B/J(B)$); in this case, the main result can be formulated as follows.

\begin{theorem}\label{t.t}
Let $B$ be a finite dimensional algebra over a field $\KK$ whose simple modules are all Schur, that is, $\End(S)=\KK$ for each simple $B$-module $S$, and let $J(B)$ be the Jacobson radical of $B$. If $A_0\subset B$ is a subalgebra such that $B=A_0\oplus J(B)$ (so $A_0\cong B/J(B)$ via the canonical projection; this exists by Wedderburn-Malcev) and if $A_0=\prod\limits_{i}M_{n_i}(\KK)$, then every maximal subalgebra of $B$ is conjugate (inside $B$) to an algebra of the following three types:\\
(a) $\left(B(k,n_i-k)\times \prod\limits_{j\neq i} M_{n_j}(\KK)\right)\oplus J(B)$, where $B(k,n_i-k)$ is the subalgebra of $M_{n_i}(\KK)$ of block upper triangular matrices with blocks of size $k$ and $n_i-k$ on the diagonal, and the parenthesis is considered as a subalgebra of $A_0$. \\
(b) $\left(\Delta^2(n_i,\KK) \times \prod\limits_{k\neq i,j} M_{n_i}(\KK)\right) \oplus J(B)$, where $n_i=n_j$ and $\Delta^2(n_i,\KK)$ is the image of the diagonal embedding $M_{n_i}(\KK)\rightarrow M_{n_i}(\KK) \times M_{n_i}(\KK)$ (here, this diagonal embedding lands in components $i,j$ of the direct product $A_0=\prod\limits_{k}M_{n_k}(\KK)$).\\
(c) $A_0\oplus H$, where $J(B)^2\subset H\subset J(B)$ and $H/J(B)^2$ is a maximal $A_0$-sub-bimodule of $J(B)/J(B)^2$.
\end{theorem} 

The theorem above follows as a consequence of Sections 2 and 3. Its version for basic algebras, stated in the language of quivers with relations, takes an even more precise form (Theorem \ref{t.basic}). As a consequence, for any finite dimensional algebra over an algebraically closed field, we find the maximal dimension $d$ of a proper subalgebra; it turns out that $d$ depends only on the smallest dimension of a simple module. The formula we obtain extends the results of \cite{Ag} to the most general case.

On the other hand, our other major motivation is representation theoretic. The study of subalgebras of certain particular classes of associative algebras is certainly not new, and is crucial in several fields: in group representation theory, for example, induction and restriction to and from subgroups is an indispensable central tool. Similarly, Hopf subalgebras of finite dimensional Hopf algebras play an important role in understanding the structure, via the Nichols-Zoeller theorem \cite{DNR}. More generally, subgroups of algebraic groups provide further examples. In fact, given a finite group $G$ and a subgroup $H$, the extension $\KK H \subseteq \KK G$ often has nice properties: it is easily seen to be separable when $H$ is a p-Sylow subgroup and ${\rm char}(\KK)=p$ (which is, essentially, Maschke's theorem), and it is also split in general, in the sense that $\KK H$ is a direct summand of $\KK G$ as $\KK H$-bimodule.  

Special classes of finite-dimensional associative algebras also possess nice subalgebras, which influence the representation theory of the full algebra. For instance, every cluster-tilted algebra can be obtained as a trivial extension of a tilted algebra \cite{AssCTA1}. Furthermore, it is known that one can obtain the tilted algebra from its cluster-tilted algebra by deleting certain arrows from the cluster-tilted algebra \cite{Ass1}. Such a process of deleting arrows has a natural interpretation as a filtration of subalgebras $C = A_0 \subset A_1 \subset \cdots \subset A_n = B$, where $C$ is tilted, $B$ is the cluster-tilted algebra corresponding to $C$, and $A_{i-1}$ is a maximal subalgebra of $A_i$ for all $i$, obtained by deleting a suitable arrow at each step. The problem of finding all such filtrations is essentially the problem of determining which tilted algebras give rise to a fixed cluster-tilted algebra, and which arrows are suitable for deletion. This problem has been solved in \cite{Ass1}, \cite{AssCTA2}. Since trivial extensions are in particular split extensions, the induction and coinduction functors between a tilted algebra and its corresponding cluster-tilted algebra also share nice properties. These have been studied in \cite{Ser}.

For general finite dimensional associative algebras, however, there does not seem to have been much work done towards understanding the representation theory via induction/restriction to subalgebras; rather, one often looks at quotient by various ideals. Perhaps the absence of a good supply of easily understood subalgebras with good properties is a reason for this. Hence, this paper is also intended to take a step in this direction; naturally, maximal subalgebras can be regarded as a first such step. Certainly, the case of induction/restriction to subalgebras can be regarded as a particular case of relating algebras via bimodules; but in general, given an algebra $A$, it is usually not straightforward to find an algebra $B$ and bimodule ${}_BM{}_A$ such that the functor $M\otimes_A(-)$ has the ``right" properties. Nevertheless, in our study, we obtain constructions that yield classes of subalgebras of associative algebras which have good representation theoretic properties and are easily described at the same time. In the last section, we provide many examples of maximal subalgebras, as well as embeddings of algebras $A\subset B$ with such relevant properties. We give examples of embeddings of quiver algebras and incidence algebras, and in particular of Dynkin quivers, and show how indecomposable representations of many ADE quivers can be obtained via induction/restriction from suitable subalgebras, which are often also ADE, thus providing relations between indecomposables of various quivers (of finite type or not). Such induction/restriction functors sometimes even produce morphisms between the representation rings, and can thus be used to relate them. While we do not attempt to create a general theory - which could go into different directions as per various types of subalgebras -  the multitude of examples and flexibility in the choices seem to suggest plenty of possible applications and a further study may be warranted.  


The paper is organized into four sections. In the first, we give background on separable functors, and related notions of split/separable extensions of algebras. In the second section, we prove several fundamental results, and define two essential types of maximal subalgebras, which we call maximal subalgebras of semisimple (or separable) type and of split type. In the third section we provide a complete classification of maximal subalgebras of semisimple algebras, and prove the full classification, and the above Theorem \ref{t.t}. Finally, in the fourth section we provide examples, illustrate the theory in several important instances, and discuss possible future investigations.  

\noindent {\bf{Acknowledgments.}} The authors would like to thank Ryan Kinser for a careful reading of a preliminary version of this paper and many useful suggestions which improved the paper; they would also and Victor Camillo encouraging discussions and suggesting a few additional references. 

\section{Split Extensions, Split-by-Nilpotent Extensions, and Separable Functors}   

\subsection{Separable Functors}
  
\noindent In this section we give some general background on separable functors, and extensions of algebras related to such functors. We will see that separable functors provide a good context in which algebras can share representation-theoretic data. Much of our exposition will follow Chapter 3 of \cite{Caen}.  

\noindent Throughout the rest of this paper, $\KK$ denotes a field, and $A \subset B$ an extension of rings. We say that $A$ is {\bf{maximal}} in $B$ if for any subalgebra $C \subset B$ such that $A \subseteq C \subseteq B$, it follows that $A = C$ or $C = B$. We denote by $J(A)$ the Jacobson radical of the algebra $A$, and $Z(A)$ the center of $A$. We will also write $C(A) = \{ b \in B \mid ba = ab$ for all $a \in A \}$ to denote the centralizer of $A$ in $B$, and $C^2(A) = \{ b \in B \mid bx = xb$ for all $x \in C(A)\}$ the bicommutant subalgebra. It is a standard fact (of Galois connections) that $C^3(A)=C(C^2(A))=C(A)$ for any $A$. Unless otherwise stated, all rings are associative $\KK$-algebras, and all $\KK$-algebras and modules are finite-dimensional over $\KK$. 

\begin{definition}  
Let $F : \mathcal{C} \rightarrow \mathcal{D}$ be a functor between categories $\mathcal{C}$ and $\mathcal{D}$. Then $F$ induces a natural transformation $\tilde{F}: \mathcal{C}(-,-) \rightarrow \mathcal{D}(F(-),F(-))$, defined by $\tilde{F}\left( X \xrightarrow[]{\alpha} Y \right) = F(X) \xrightarrow[]{F(\alpha)} F(Y)$. $F$ is called {\bf{separable}} if $\tilde{F}$ admits a natural section, i.e. a natural transformation $G : \mathcal{D}(F(-),F(-)) \rightarrow \mathcal{C}(-,-)$ with $G\circ \tilde{F} = 1_{\mathcal{C}(-,-)}$.
\end{definition}



\noindent Separable functors were first studied in \cite{Nast}. The terminology comes from the fact that an extension of rings $\varphi : A \rightarrow B$ is separable if and only if the restriction functor $\Res_{\varphi} $ is separable (Prop. 1.3.1 of \cite{Nast}.) Such functors have found applications in representation theory as a general setting for Maschke-type theorems \cite{Caen}. For the purposes of this paper, we need only a few essential facts, which we list and recall below without proof. 

\bigskip




\begin{theorem}[Rafael] 
Let $ F: \mathcal{C}\rightarrow \mathcal{D}$ have a right adjoint $G : \mathcal{C} \rightarrow \mathcal{D}$.  

\begin{enumerate} 
\item $F$ is separable if and only if the unit $\eta : 1_{\mathcal{C}} \rightarrow GF$ of the adjunction $(F,G)$ splits, in the sense that there is a natural transformation $\nu : GF \rightarrow 1_{\mathcal{C}}$ such that $\nu \circ \eta = 1_{1_{\mathcal{C}}}$, the identity transformation of $1_{\mathcal{C}}$.
\item $G$ is separable if and only if the counit $\e : FG \rightarrow 1_{\mathcal{D}}$ of the adjunction $(F,G)$ cosplits, in the sense that there is a natural transformation $\zeta : 1_{\mathcal{D}} \rightarrow FG$ such that $\e \circ \zeta = 1_{1_{\mathcal{D}}}$, the identity transformation of $1_{\mathcal{D}}$.
\end{enumerate} 
\end{theorem} 

\begin{proof} 
See Ch. 3.1, Thm. 24 of \cite{Caen}.
\end{proof}

\noindent The following are straightforward corollaries to the general theory of separable functors, known to specialists, and they provides us with motivation for considering split and separable extensions. Such results also arise in the context of separable bimodules \cite{Caen2}. Recall that an algebra is representation finite if, up to isomorphism, there are only finitely many left (equivalently, right) indecomposable $A$-modules (in which case, every module is a direct sum of indecomposable modules). For this and other notions in the representation theory of finite dimensional algebras we refer to the well known textbooks \cite{Ass5, ARS}. In what follows, we let $(F:A{\rm-mod}\longrightarrow B{\rm-mod}, G:B{\rm-mod}\longrightarrow A{\rm-mod})$ be an adjoint pair between categories of finite dimensional modules over algebras $A$ and $B$, respectively. Denote $\operatorname{Ind}(A),\operatorname{Ind}(B)$ the sets of isomorphism classes of $A$-, and respectively, $B$-modules. We include a brief argument only  as an illustration of the theory; it also follows readily as consequence of more general results on separable bimodules and functors; see e.g. \cite{Caen,Caen2}.

\begin{lemma}\label{l.1.1}
(i) If $F$ is separable, then for each $X\in \operatorname{Ind}(A)$, there is $Y\in \operatorname{Ind}(B)$ such that $X$ is a direct summand in $G(Y)$. In particular, if $B$ is representation finite, then so is $A$. \\
(ii) Moreover, if the unit $\eta$ is an isomorphism, 
for every $X\in \operatorname{Ind}(A)$, $F(X)=Y\oplus\bigoplus\limits_{i}Y_i$ where $Y,Y_i$ are indecomposables, $X=G(Y)$ and $G(Y_i)=0$ for all $i$. If, in addition, $G$ is faithful, then $F$ induces an injective map from $\operatorname{Ind}(A)$ to $\operatorname{Ind}(B)$.
\end{lemma}
\begin{proof}
If $X\in \operatorname{Ind}(A)$, then Rafael's Theorem implies that $X$ is a direct summand of $GF(X)$ (the unit $\eta$ map splits); decomposing $F(X)=\bigoplus\limits_i Y_i$ as a finite direct sum of indecomposables yields $X\stackrel{\oplus}{\hookrightarrow}\bigoplus\limits_{i}G(Y_i)=GF(X)$. So $X\stackrel{\oplus}{\hookrightarrow}G(Y_i)$ for some $i$. Moreover, if $\eta$ is an isomorphism, then it follows that all but one term $G(Y_0)$ in the direct sum decomposition $X=\bigoplus\limits_{i}G(Y_i)$ are zero, and $G(Y_0)=X$. If $G$ is faithful, the last part follows from this.
\end{proof}

\noindent Of course, the previous Lemma has a similar version for the right adjoint:

\begin{lemma}\label{l.1.2}
(i) If $G$ is separable, then for each $Y\in \operatorname{Ind}(B)$, there is $X\in \operatorname{Ind}(A)$ such that $Y$ is a direct summand in $F(X)$. In particular, if $A$ is representation finite, then so is $B$. \\
(ii) Moreover, if the counit $\epsilon$ is an isomorphism, 
for every $Y\in \operatorname{Ind}(B)$, $G(Y)=X\oplus\bigoplus\limits_{j}X_j$ where $X,X_j$ are indecomposable, $Y=F(X)$ and $F(X_j)=0$ for all $j$. If, in addition, $F$ is faithful, then $G$ induces an injective map from $\operatorname{Ind}(B)$ to $\operatorname{Ind}(A)$.
\end{lemma}


\subsection{Localization} 

We remark how the localization or ``corner rings" also produce examples of such separable functors. If $A$ is any ring, and $e$ an idempotent in $A$, we have the following adjoint pairs:

$$eAe{\rm-Mod}\stackrel{L}{\longrightarrow} A{\rm-Mod} \stackrel{E}{\longrightarrow} eAe{\rm-Mod} \stackrel{R}{\longrightarrow} A{\rm-Mod}$$

where $L(N)=Ae\otimes_{eAe} N$; $E(M)=eM$; $R(N)=\Hom_{eAe}(eA,N)$. Then $E$ can be regarded  as a localization functor in the sense of P. Gabriel \cite{G}. The functors $(L,E)$ and $(E,R)$ form adjoint pairs. It is well known (and not difficult to see) that $E\circ L={\rm Id}$ and $E\circ R={\rm Id}$ via the unit, and respectively, counit of the adjunction. These identities also imply that $L$ and $R$ are faithful. 

\begin{corollary}\label{c.1.1}
The functors $L,R$ are separable. If $A$ is a finite dimensional algebra, then each indecomposable $eAe$-module $N$ is of the form $N=eM$ for an indecomposable module $M$; moreover, one can pick such $M$ in two ways: (1) such that $Ae\otimes_{eAe} N=M\oplus \bigoplus\limits_i M_i$, a direct sum of indecomposable $A$-modules with $eM_i=0$ or (2) such that $\Hom_{eAe}(eA,N)=M\oplus \bigoplus\limits_i M_i$ a direct sum of indecomposable $A$-modules with $eM_i=0$. 
\end{corollary}

\noindent We note that results as the previous one can be obtained in more general context of localization of categories, as in many instances such localizations have left and right adjoints which are ``one-sided" inverses, hence satisfying such properties \cite{G}.

\subsection{Split and separable extensions}

\noindent For any morphism of algebras $\varphi : A \rightarrow B$, we have an adjoint pair of functors $\left( (-)\otimes_BB, \Res_{\varphi} \right)$ between right $A$-modules and right $B$-modules given by induction and restriction along $\varphi$. For the purposes of this paper we may assume that $\varphi$ is injective, identify $A$ with $\varphi (A)$, and only consider the inclusion $i: \varphi (A) \rightarrow B$; however, the statements below hold in general. 

\begin{definition} 
 Let $A \subset B$ be an extension of algebras. $B$ is called a {\bf{split}} extension of $A$ if $(-)\otimes_{A}B$ is separable, and {\bf{separable}} if $\Res_{A}^B = \Res_i$ is separable. 
\end{definition}   

\noindent  The functor characterizations of separable and split extensions are often useful. However, it is usually easier to use the following criterion to directly verify that a given extension of algebras is split or separable:

\begin{lemma} 
Let $A\subset B$ be an extension of algebras. Then the following hold: 
\begin{enumerate} 
\item $B$ is a split extension of $A$ if and only if there is an $A$-sub-bimodule $I$ of $B$ such that $B = I \oplus A$ as $A$-bimodules. 
\item $B$ is a separable extension of $A$ if and only if the multiplication map $u : B\otimes_{A}B \rightarrow B$ is a split epimorphism of $B$-bimodules, if and only if there is an element $e \in B\otimes_{A}B$ such that $u(e) = 1$ and $eb = be$ for all $b \in B$.
\end{enumerate}
\end{lemma}  

\begin{proof} 
See \cite{Caen2}, \cite{Kad}.
\end{proof} 

\noindent We note that another notion of split extensions, in the spirit of Hochschild, is often present in literature, which requires the subspace $I$ above to be an ideal.

\begin{corollary}  \label{c.1.2}
Let $A \subset B$ be an extension of algebras. Then the following hold: 
\begin{enumerate} 
\item If $B$ is a split extension of $A$, then every indecomposable $A$-module is a direct summand of a restriction of an indecomposable $B$-module; in particular, if $B$ is representation-finite, then so is $A$. 
\item If $B$ is a separable extension of $A$, then every indecomposable $B$-module is a direct summand of a module induced from an indecomposable $A$-module; in particular, if $A$ is representation-finite, then so is $B$.
\end{enumerate}
\end{corollary}

\noindent Hence split extensions hence allow one to transfer representation-theoretic properties of $B$ down to $A$, whereas separable extensions allow one to transfer such data from $A$ to $B$. To close this section, we recall some common representation-theoretic terminology:

\begin{definition} \label{d.split}
Suppose $A\subset B$ is a split extension, with $B = I\oplus A$ as before. Then $B$ is {\bf{split-by-nilpotent}} if $I$ is a nilpotent ideal of $B$.
\end{definition}   

\begin{definition}\label{d.trivial}
Suppose $A\subset B$ is a split-extension, with $B = I\oplus A$ as before. If $I^2 = 0$, we say that $B$ is the {\bf{trivial extension}} of $A$ through the $A$-bimodule $I$.
\end{definition}

\noindent Split-by-nilpotent extensions have a fairly rich theory. In particular, trivial extensions arise prominently as cluster tilted algebras (see \cite{AssCTA1}, \cite{AssCTA2}, \cite{Buan}, \cite{Platz}, \cite{Ser}.) We will see in the next section that, under suitable hypotheses, we are able to gather partial information about the maximal subalgebras of a fixed algebra $B$ by considering related trivial extensions. This reduction works for general $B$, but relies heavily on the ideals shared by $B$ and its maximal subalgebras. For more on split-by-nilpotent extensions in general, see \cite{Ass1}, \cite{Ass4},  \cite{Ass2}, \cite{Ass3}, \cite{Caen2}, \cite{Platz}, \cite{Ser}.

\section{General Results}    \label{s.gen}

\noindent We start off by recalling a simple lemma related to primitive orthogonal idempotents. It appears as Theorem 10.3.6 in \cite{Haz}, but does not appear to be well-known. Hence, we reproduce the brief proof here: 

\begin{lemma} 
Let $A$ be a ring, $e_1A\oplus \cdots \oplus e_nA = A = f_1A\oplus \cdots \oplus f_nA$ a decomposition of $A_A$ into right ideals with $\{ e_1,\ldots , e_n\}$ and $\{ f_1,\ldots , f_n \}$ two collections of pairwise orthogonal idempotents which sum to $1$. Furthermore suppose that $e_iA \cong f_iA$ as $A$-modules. Then there is an invertible element $a \in A$ such that $f_i = ae_ia^{-1}$. In particular, if $\{e_1,\dots,e_n\}$ and $\{f_1,\dots,f_n\}$ are two complete systems of primitive orthogonal idempotents, then $f_{\sigma(i)} = ae_ia^{-1}$ for a suitable permutation $\sigma$.
\end{lemma} 

\begin{proof} 
We know that $\Hom_A(e_iA, f_iA) \cong f_iAe_i$, with homomorphisms in $\Hom_A(e_iA, f_iA)$ given by left multiplication by elements in $f_iAe_i$; hence, the isomorphism $e_iA \cong f_iA$ is realized by left multiplication by some $a_i \in A$ with $f_ia_i = a_ie_i = a_i$. Set $a = \sum_{i=1}^n{a_i}$. Let $b_i \in e_iAf_i$ realize the inverse isomorphism to $a_i$. Then $a_ib_i=f_i$ and $b_ia_i=f_i$. Set $b = \sum_{i=1}^n{b_i}$. Then $ab = \sum_i{a_ib_i} = \sum_i{f_i} = 1$ and $ba = \sum_i{b_ia_i} = \sum_i{e_i} = 1$, so $a$ is a unit. Since $f_ia = f_ia_i = a_ie_i = a e_i$, we have $f_i = ae_ia^{-1}$ for all $i$. 
\end{proof} 

\begin{corollary} 
Let $A \subset B$ be an extension of finite-dimensional $\KK$-algebras, and let $\{ e_1,\ldots , e_n \}$ be a complete collection of primitive orthogonal idempotents for $B$. Then there is a unit $u \in U(B)$ such that the subalgebra $uAu^{-1} \subset B$ contains a complete collection of primitive orthogonal idempotents for $uAu^{-1}$, whose elements are sums of elements in $\{ e_1,\ldots , e_n \}$.
\end{corollary} 

\begin{proof} 
Let $\{ \e_1,\ldots , \e_k \}$ be a complete collection of primitive orthogonal idempotents for $A$. Then in $B$, the $\e_i$'s are still a complete collection of orthogonal idempotents, and each $\e_i$ can be written as a sum of primitive orthogonal idempotents, say $\e_i = \sum_{j=1}^{n_i}{f_j^{(i)}}$. By the previous lemma, there is a unit $u \in U(B)$ such that $\{uf_i^{(j)}u^{-1}\}$ form a permutation of $\{ e_1,\ldots , e_n \}$ for all $i$ and $j$, from which the claim follows.
\end{proof}

\noindent {\bf{Note:}} Automorphisms of $B$ preserve maximality of subalgebras. In other words, if $\alpha$ is an automorphism of $B$ and $A$ is a maximal subalgebra, then $\alpha (A) \subset \alpha (B) = B$ is another maximal subalgebra. More generally, $\KK$-algebra endomorphisms of $B$ induce order-preserving maps on the poset of subalgebras of $B$ (ordered with respect to inclusion.) Often, it will be natural to classify maximal subalgebras of $B$ up to conjugation by a unit (i.e. up to orbits of the action of inner automorphisms on $B$.)
\bigskip  

\noindent {\bf{Note:}} In general, it is not true that if $A$ and $A'$ are isomorphic subalgebras of $B$, that $A' = \alpha (A)$ for some automorphism $\alpha : B \rightarrow B$. For instance, if $B = \mathbb{F}_2[X,Y]/(X^2, Y^4)$, then $A = \Sp_{\mathbb{F}_2}\{1, X\} \cong \Sp_{\mathbb{F}_2}\{ 1, Y^2 \} = A'$, but $\alpha (A) \neq A'$ for any automorphism $\alpha$ of $B$. Indeed, any such $\alpha$ would necessarily map $X$ to $Y^2$. Since $Y^2 \neq 0$, $\alpha (Y) = \alpha_XX + Y + \mathcal{O}(XY)$, for some element $\alpha_X \in \mathbb{F}_2$, and where $\mathcal{O}(XY)$ is a multiple of $XY$. But $XY^2 \neq 0$ in $B$, and $\alpha (X) \alpha (Y)^2 = Y^2 (\alpha_X X + Y + \mathcal{O}(XY))^2 = Y^2\cdot Y^2 = 0$, a contradiction. In section 4, we will see more examples of this behavior.

\begin{lemma} \label{l.1}
Let $A$ be a maximal subalgebra of $B$. Then $A+ J(A)\cdot B = A = A+ B\cdot J(A)$, and exactly one of the following holds: \par
(i) $J(B) \subset A$ and $A/J(B)$ is a maximal subalgebra of the semisimple algebra $B/J(B)$.\par
(ii) $J(B)\not\subset A$, in which case $J(A) = A \cap J(B)$ and $A$ and $B$ have the same simple modules, that is, the functor ${\rm Res}_A^B$ induces a bijection from simple $B$-modules to simple $A$-modules.
\end{lemma} 

\begin{proof} 
Note that  $A+J(A)\cdot B$ is a subalgebra with $A \subseteq A + J(A)\cdot B \subseteq B$, and hence, either $A = A+J(A)\cdot B$ or $A+J(A)\cdot B = B$. If the latter equality holds, then as left $A$-modules, $J(A)\cdot (B/A) = (J(A)\cdot B+A)/A = B/A$, and since $A$ and $B$ are finite-dimensional, we get $B/A = 0$ by Nakayama's Lemma. But then $A= B$, contradicting the properness of the inclusion $A\subset B$. Hence, $A+J(A)\cdot B = A$, and the proof of $A+B\cdot J(A) = A$ is similar. \\
Now, since $A \subset A+J(B) \subset B$ and $A+J(B)$ is a subalgebra, either $A = A+J(B)$ or $A+J(B) = B$. If $A = A+J(B)$, then in particular $J(B) \subset A$, and by correspondence, in this case $A/J(B)$ is a maximal subalgebra of $B/J(B)$. Otherwise, $A+J(B) = B$. In this case $B/J(B) = (A+J(B))/J(B) \cong A/(J(B)\cap A)$ as $A$-modules. In fact, this is an isomorphism of $\KK$-algebras, and since $A/(J(B)\cap A) \cong B/J(B)$ is semisimple, $J(A) \subset J(B)\cap A$. But since $J(B)\cap A$ is a nil (and nilpotent) ideal of $A$ (since $J(B)$ is so as $\dim(B)<\infty$), we also have $J(B)\cap A \subset J(A)$. Therefore, $B/J(B) \cong A/(J(B)\cap A) = A/J(A)$; it is easy to see now that this implies that $\Res_A^B$ induces a bijection from simple $B$-modules to simple $A$-modules. 
\end{proof}

This Lemma breaks maximal subalgebras into two possible types: those subalgebras $A$ of $B$ corresponding to maximal subalgebras $A'$ of the semisimple algebra $B/J(B)$, and the rest, which satisfy $J(A)=A\cap J(B) \neq J(B)$. The first type reduces to the study of the maximal subalgebras of semisimple algebras, and we call them {\bf{maximal subalgebras of semisimple type}} or {\bf{maximal subalgebras of separable type}}, for reasons that will be apparent later. We call the maximal subalgebras of the second kind {\bf{maximal subalgebras of split type}}. We construct a large class of subalgebras of the split type in the following main example; this class essentially produces all examples in many cases, as will be shown next.  

\begin{example}[Maximal subalgebras of split type.]\label{e.split}
(1) Let $H$ be a two-sided ideal of $B$, properly contained in $J(B)$, and maximal with this property (that is, a maximal $B$-sub-bimodule of $J(B)$), and such that the projection $B/H \rightarrow B/J(B)$ admits an algebra retract; equivalently, there is a subalgebra $A'\subseteq B/H$ such that $B/H= A'\oplus J(B)/H$ is a trivial extension, and also a Hochschild split extension, with the ideal $I=J(B)/H$ satisfying $I^2=0$. Let $A$ be a subalgebra of $B$ containing $H$, such that $A/H=A'$. Then $A$ is a maximal subalgebra of $B$, which we will say is {\bf of split type}. Indeed, this follows if we show that $A'$ is a maximal subalgebra of $B'=B/H$. Note that $B'=A'\oplus I$ as $A'$-bimodules; $I$ is a simple $B$-bimodule, and therefore a simple $B'$-bimodule. Since $I^2=0$, the $I$-part of $B'$ acts trivially on itself, which implies that $I$ is a simple $A'$-bimodule. Now, if $A'\subsetneq A''\subseteq B'$ is an intermediate subalgebra, then $A''$ is an $A'$-sub-bimodule, and since $B/A'$ is a simple $A'$-bimodule, the conclusion follows. \\
(2) Assume now that $B/J(B)$ is separable over $\KK$. By the Wedderburn-Malcev theorem, there exists a subalgebra $A_0$ of $B$ such that $B=A_0\oplus J(B)$. If $H$ is a maximal two-sided ideal contained in $J(B)$, then as in (1), it follows that $A_0\oplus H$ is a maximal subalgebra of $B$.  In particular, this is true when $\KK$ is perfect (so when it is algebraically closed), or when the algebra is Schur (i.e. $\End(S)=\KK$ for every simple module), and so when the algebra is basic pointed (that is, every simple module is 1-dimensional). 
\end{example}

The next theorem gives the complete general structure of maximal subalgebras. We will need to use the following well known fact: if $M$ is a simple $B$-bimodule which is finite dimensional (or, more generally, artinian or semiartinian both as a left and as a right module over $B$), then $M$ is semisimple as left and as right $B$-module (and even isotypical). Indeed, taking $M_0$ to be the socle of ${}_BM$ and $b\in B$, then $M_0b$ is semisimple since it is a quotient of the left module $M_0$, so $M_0b\subseteq M_0$ and so $M_0$ a sub-bimodule. Since $M_0\neq 0$, $M_0=M$.

\begin{theorem}\label{t.gen}
Let $B$ be a finite dimensional $\KK$-algebra. Then any maximal subalgebra of $A$ is either of semisimple type, or of split type. Moreover, if $B/J(B)$ is a separable $\KK$-algebra, and $A_0$ is a subalgebra of $B$ with $A_0\oplus J(B)=B$ (which exists by Wedderburn-Malcev Theorem), then:\\
(i) Every maximal subalgebra $A$ of $B$ of split type is conjugate to one of the form in \ref{e.split} (2); that is, $A=u(A_0\oplus H)u^{-1}$, where $u\in U(B)$,  and $H\subset J(B)$ is a two-sided ideal of $B$ maximal inside $J(B)$. \\
(ii) Every maximal subalgebra $A$ of $B$ of semisimple type is of the form $A=A'\oplus J(B)$, where $A'$ is a maximal subalgebra of $A_0$.
\end{theorem}
\begin{proof}
As in Lemma \ref{l.1}, if $J(B)\subseteq A$ then $A$ is of semisimple type. Otherwise, $J(B)\cap A=J(A)$. Let $K$ be the largest (two-sided) ideal of $B$ contained in $A$, and let $I\supset K$ be an ideal of $B$, minimal over $K$ (such an ideal exists since $0\neq B/K$ is finite dimensional). Then $I/K$ is a simple $B$-bimodule, and thus semisimple as a left and right $B$-module. By Lemma \ref{l.1}, it is also semisimple as a left and right $A$-module. Also, $K\subseteq A\cap I$, and so $I/A\cap I$ is left and right $A$-semisimple (it is a quotient of $I/K$ as $A$-bimodules). Note that $A+I=B$, since it is an intermediate subalgebra $A\subsetneq A+I\subseteq B$.\\
We have the following isomorphisms of $A$-bimodules
$$\frac{J(B)}{J(A)} \cong \frac{J(B)}{A\cap J(B)} \cong \frac{A+J(B)}{A} \cong \frac{B}{A} \cong \frac{A+I}{A} \cong \frac{I}{A\cap I} $$
Hence, $J(B)/J(A)$ is semisimple both as a left and right $A$-module. Therefore, $J(A)\cdot \frac{J(B)}{J(A)}=0$ so $J(A)\cdot J(B)\subseteq J(A)$. Similarly, $J(B)\cdot J(A)\subseteq J(A)$. But then $B\cdot J(A)=(A+J(B))\cdot J(A)= J(A)+J(B)\cdot J(A)=J(A)$ and similarly $J(A)\cdot B=J(A)$, which shows that $J(A)$ is in fact a two sided ideal of $B$. \\
Since $I/K$ is a simple $B$-bimodule, $B=A+J(B)$, and $J(B)$ acts trivially on $I/K$, it follows that $I/K$ is simple as an $A$-bimodule. Also, $I/A\cap I$ is a non-zero $A$-bimodule quotient of $I/K$, and so it is a simple $A$-bimodule. By the above isomorphism, so is $J(B)/J(A)$. Since $A\subset B$, $J(B)/J(A)$ is a simple $B$-bimodule as well. Therefore, the semisimplicity of $J(B)/J(A)$ as a left and right $B$-module also implies that $J(B)\cdot \frac{J(B)}{J(A)}=0$ so $J(B)^2\subseteq J(A)$. \\ 
Finally, denote $H=J(A)$. By the above, $H$ is an maximal $B$-sub-bimodule of $J(B)$. It is clear that $A'=A/H$ is a semisimple subalgebra of $B'=B/H$, and that $B'=A\oplus I$ where $I=J(B)/H$ has $I^2=0$, and this is a trivial extension. \\
Furthermore, assuming that $B/J(B)$ is separable, by the Wedderburn-Malcev theorem there is a sub-algebra $A_0$ of $B$ with $B=A_0\oplus J(B)$ (see e.g. \cite[Chapter 11]{Pi}) and by the same theorem, in the case when $A'$ is of split type, $A'=A/H$ and $(A_0\oplus H )/ H$ are conjugate inside $B'=B/H$ by some invertible $\overline{u}=u+H\in B/H$. Lifting back, this easily implies that $A$ and $A_0\oplus H$ are conjugate too.\\
If $A$ is of semisimple type, then the maximal subalgebras of $B/J(B)$ are in 1-1 correspondence with those of $A_0$, since the composition of maps $A_0\subset B\rightarrow B/J(B)$ is an isomorphism of algebras $A_0\cong  B/J(B)$. If $A$ is maximal of semisimple type, it contains $J(B)$, and it follows from this correspondence that $A/J(B)$ has the form $A/J(B)=A'\oplus J(B)/J(B)$ for some maximal subalgebra $A'\subset A_0$, so $A=A'\oplus J(B)$. But, within $A_0$, by Lemma \ref{l.3.1}, Theorem \ref{t.3.1} and Remark \ref{r.ss}, every such maximal subalgebra $C$ is conjugate to one of this form $A'$ by some invertible element $a\in A_0$; obviously, then the algebra $C\oplus J$ is conjugate to $A'\oplus J$ via $a$.
\end{proof}

In the particular case when $\KK$ is algebraically closed or at least each simple module $S$ is Schur ($\End(S)=\KK$, i.e. $\KK$ is a splitting field for $B/J(B)$), this theorem will take more precise forms. 

\begin{remark}\label{r.practical}
We note that, if $H$ is a maximal sub-bimodule of $J(B)$, then $J(B)/H$ is semisimple both as a left and as a right $B$-module, and so $J(B)^2\subseteq H$. Hence, in order to find such $H$, one needs to determine a maximal $B$-sub-bimodule (or, equivalently, $B/J(B)$-sub-bimodule) of $J(B)/J(B)^2$ and pull back in $B$. In the case of particular interest when $B/J(B)$ is separable, and $A_0$ is a subalgebra of $B$ with $B=A_0\oplus J$ as before, then to build maximal subalgebras of semisimple type we need to determine an $A_0$-sub-bimodule $H'$ of $J(B)/J(B)^2$, and then let $H$ be such that $H/J(B)^2=H'$. It is easy to note that such $H$ is in fact a two sided ideal of $B$. In particular, if the endomorphism ring of each simple $B$-module is $\KK$, then one can write $J(B)=J(B)^2\oplus T$ as $A_0$-bimodules, and one only needs to determine the maximal $A_0$-sub-bimodules of $T$. Throughout the rest of this paper, we say that $B$ is a {\bf{Schur}}, or {\bf{Schurian}}, $\KK$-algebra if $\End_A(S) = \KK$ for all simple $A$-modules $S$.
\end{remark}

\section{Semisimple Classification}   \label{s.3} 

\subsection{The Simple Case}

Example \ref{e.split} show us how maximal subalgebras of split type arise in general. In turn, Theorem \ref{t.gen} and Remark \ref{r.practical} show us that under relatively weak conditions, the problem of constructing them is equivalent to the problem of constructing maximal sub-ideals of the Jacobson radical. However, at this stage we have comparatively little information on maximal subalgebras of ``semsimple'' or ``separable'' type. To completely understand the classification of maximal subalgebras, we will need to further examine these subalgebras. Essentially, this means understanding maximal subalgebras of semisimple algebras. 

Some of our results hold in extra generality, and wherever possible, we have tried to state them under the weakest conditions possible. Before stating these results, however, it will be convenient to fix some notation for certain important matrix subalgebras:

\begin{definition}\label{d.but} 
Let $R$ be an arbitrary (non-commutative) ring, $n$ a natural number, and $\lambda = (\lambda_1,\ldots , \lambda_t)$ a composition of $n$, i.e. a collection of non-negative integers $\lambda_i$ such that $\lambda_1 + \ldots + \lambda_t = n$. Throughout, assume that $\lambda_i > 0$ for all $i$. Given such $\lambda$, one can decompose each $X \in M_n(R)$ as $X = (X_{ij})$, where $X_{ij}$ is a $\lambda_i\times\lambda_j$-matrix with entries in $R$. Let $B_{\lambda}(R)$ denote the collection of all such $X$ with $X_{ij} = 0$ whenever $j < i$. If the ring $R$ is understood, we write $B(\lambda) = B_{\lambda }(R)$. We call $B(\lambda )$ the ring of {\bf{block upper-triangular matrices}} corresponding to the composition $\lambda$.
\end{definition}   

\noindent {\bf{Note:}} As is convenient, we will use the frequency notation for a composition $\lambda$, i.e. we will write $\lambda = (1^{a_1} 2^{a_2}\cdots )$, where $a_k$ denotes the cardinality of the set $\{ i \in \NN \mid \lambda_i = k\}$. 
\bigskip

\noindent  As one would expect, the classification maximal subalgebras of semisimple algebras depends in good part on the maximal subalgebras of simple algebras. We will now construct the main classes of such algebras; the terminology chosen (type numbering) is done such that it agrees with that of \cite{Rac0}; there will be one additional new type of example.

\begin{example}[Maximal subalgebras of semisimple type 1 (or type S1)]
Let $D$ be a division ring ($\KK$-algebra), $n\geq 2$, and $1\leq k<n$. Then the $\KK$-subalgebra of block triangular matrices $B_{(k,n-k)}(D)$ is a maximal subring ($\KK$-subalgebra) of $M_n(D)$. Indeed, if $A'$ is an intermediate subalgebra and $x\in A'-B_{(k,n-k)}(D)$, then by subtracting a suitable block triangular matrix we may assume $x=\left(\begin{array}{cc} 0 & 0 \\ T & 0  \end{array}\right)$ with $T$ an $(n-k)\times k$ non-zero block. Then for $X\in M_{n-k}(D), Y\in M_k(D)$, 
$$\left(\begin{array}{cc} 0 & 0 \\ 0 & X  \end{array}\right)\left(\begin{array}{cc} 0 & 0 \\ T & 0  \end{array}\right)\left(\begin{array}{cc} Y & 0 \\ 0 & 0  \end{array}\right)=\left(\begin{array}{cc} 0 & 0 \\ XTY & 0  \end{array}\right)\in A'$$
But it is well-known that $M_{n-k,k}(D)$ is a simple $M_{n-k}(D)$-$M_{k}(D)$-bimodule, so $\left(\begin{array}{cc} 0 & 0 \\ * & 0  \end{array}\right)\subseteq A'$ which implies $A'=M_n(D)$. \\
We remark that these subalgebras were considered in \cite{Rac0,Rac} where it was proved they are maximal $Z(D)$-subalgebras; here we need that they are also maximal as $\KK$-subalgebras (and, in fact, they are maximal subrings).
\end{example}

\begin{example}[Maximal subalgebras of semisimple type 2 (or type S2)]
Let $D$ be a division ring ($\KK$-algebra), $n\geq 1$ and $Z(D)\subseteq F$ a minimal field extension contained in $M_n(D)$ (i.e. there are no intermediate subfields). Then its centralizer $C(F)$ is a maximal subring ($\KK$-subalgebra) of $M_n(D)$. Indeed, $C(F)$ contains $Z(D)$ so it is a $Z(D)$-subalgebra, and by \cite{Rac0}, it is a maximal subalgebra of $M_n(D)$. But any intermediate subring (or $\KK$-subalgebra) $C(A)\subseteq A' \subseteq M_n(D)$ contains $Z(D)$ and is a $Z(D)$-subalgebra, so maximality as a subring (resp. $\KK$-subalgebra) follows from maximality as a $Z(D)$-subalgebra. We note that in this case, as it is well known for simple subalgebras of central simple algebras, the algebra $C(F)$ is a simple $Z(D)$-subalgebra of $M_n(D)$ and $C^2(F)=F$.
\end{example}

\noindent $M_n(D)$ admits a third type of maximal subalgebra, not present in \cite{Rac0}, \cite{Rac}. Building up to it, we need a few remarks.

\begin{proposition}
Suppose $A$ is a maximal subring (subalgebra) of the algebra $B$, and $n\geq 1$. Then $M_n(A)$ is a maximal subring (subalgebra) of the algebra $M_n(B)$.
\end{proposition}
\begin{proof}
Let $X=(x_{ij})\notin M_n(A)$; we prove that the subring $A'$ generated by $M_n(A)$ together with $X$ is $M_n(B)$. There is some entry $x_{kl}=x\notin A$. Multiplying by appropriate matrix units $E_{ij}$ (which are in $M_n(A)$) we can assume $E_{kl}x\in A'$, and using appropriate permutation matrices, that $E_{11}x\in A'$. But $A\cdot E_{11}\subset A'$, and since the subring of $B$ generated by $A$ and $x$ is $B$, it follows that $B \cdot E_{11}\subseteq A'$. Again using permutation matrices, $B\cdot E_{ij}\subseteq A'$ for all $i,j$, which then implies the claim.
\end{proof}

\begin{example}[Maximal subalgebras of semisimple type 3 (or type S3)] 
Let $F\subseteq L$ be a minimal field extension, $n\geq 1$ and $D$ a central division $F$-algebra. Then the algebra (ring) $M_n(D)=F\otimes_F M_n(D)$ is a maximal $F$-subalgebra (subring) of $L\otimes_F M_n(D)$. First, since $F=F\otimes_F F$ is contained in $F\otimes_F M_n(D)$, we may work with subalgebras instead of subrings. The commutative diagram
$$\xymatrix{F\otimes_F M_n(D)\ar[r]^{\subseteq}\ar[d]_{\cong} & L\otimes_F M_n(D)\ar[d]^{\cong} \\
M_n(D) \ar[r]_{\subseteq} & M_n(L\otimes_F D)}$$
shows that it is enough to prove the minimality of the extension $M_n(D)\subseteq M_n(L\otimes_F D)$, and by the previous proposition, it suffices to show that $D=F\otimes_F D$ is maximal in $L\otimes_F D$. Consider $F\otimes_F D\subsetneq A'\subsetneq L\otimes_F D$ and let $x\in A'\setminus F\otimes_F D$ be an element of minimal tensor rank among elements in $A'\setminus F\otimes_F D$. Write $x=\sum\limits_{i=1}^n\lambda_i \otimes_F b_i$, where $n$ is the tensor rank of $x$, and assume $n>1$. Then the $\lambda_i$'s, and the $b_i$'s, are each linearly independent over $F$. By multiplying $x$ with $1\otimes_F b_1^{-1}$, we may assume $b_1=1\in D$, and $x=\lambda_1\otimes_F 1+\sum\limits_{i>1}\lambda_i\otimes_F b_i$. \\
For $b\in D$, the commutator $[1\otimes b,x]=\sum\limits_{i>1}\lambda_i\otimes_F [b,b_i]$ is in $A'$ and has tensor rank $<n$. Hence, $[1\otimes b,x]=1\otimes \varphi(b)$ for some uniquely determined element $\varphi(b)\in D$ (which depends on $b$). But since $\lambda_i$ are linearly independent over $F$, by basic tensor linear algebra, the equality $\sum\limits_{i>1}\lambda_i\otimes_F [b,b_i]=1\otimes \varphi(b)$ implies that for $i>1$ there are $\mu_i(b)\in F$ such that $[b,b_i]=\mu_i(b)\varphi(b)$ and $\sum\limits_{i>1}\lambda_i\mu_i(b)=1$. Again since the $\lambda_i$ are independent, such $\mu_i(b)$ are uniquely determined and they will not depend on $b$. Hence, we have $[b,b_i]=\mu_i \varphi(b)$. Note that $\mu_i\neq 0$ - otherwise, $[b,b_i]=0$ for all $b$ implies $b_i\in F$ ($D$ is central), so $b_i$ and $b_1$ are not independent, a contradiction. Also, $[b,\frac{1}{\mu_i}b_i-\frac{1}{\mu_j}b_j]=0$ for all $i,j$ and all $b\in D$, which again implies $\frac{1}{\mu_i}b_i-\frac{1}{\mu_j}b_j\in F$. Write $\frac{1}{\mu_i}b_i=\frac{1}{\mu_2}b_2+f_i$, $f_i\in F$. Then
\begin{eqnarray*} 
x & = & \lambda_1\otimes_F 1+\sum\limits_{i>1}\lambda_i\mu_i\otimes_F \frac{1}{\mu_i}b_i =\lambda_1\otimes_F 1+\sum\limits_{i>1}\lambda_i\mu_i\otimes_F (\frac{1}{\mu_2}b_2+f_i)\\
& = & \lambda_1\otimes_F 1+(\sum\limits_{i>1}\lambda_i\mu_i)\otimes_F \frac{1}{\mu_2}b_2+(\sum\limits_{i>2}\lambda_i\mu_if_i)\otimes_F 1 = \lambda\otimes_F 1+ 1\otimes_F a
\end{eqnarray*}
for suitable $\lambda\in L$ and $a\in D$. In particular, this means that $n=2$ to begin with. But $1\otimes_F a\in F\otimes_F D\subset A'$, so $\lambda\otimes_F 1\in A'-F\otimes_F D$, which contradicts the original assumption that $n>1$. It follows that $x$ has the form $x=q\otimes_F c$ to begin with, and then $x(1\otimes_F c^{-1})=q\otimes_F 1\in A'-F\otimes_F D$. We deduce that $q\notin F$, and so the set $L'=\{\lambda\in L\,|\,\lambda\otimes_F 1\in A'\}$ strictly contains $F$; it is also a subfield of $L$, and so $L'=L$ by the minimality of $F\subset L$. Finally, $L\otimes_F F\subseteq A'$ and $F\otimes_F D\subseteq A'$ implies $A'=L\otimes_F D$, and the proof is finished.
\end{example}

\begin{lemma} \label{l.3.1}
Let $D$ be a division ring containing $\KK$ with $\dim_{\KK}D < \infty$, and let $A \subset M_n(D)$ be a maximal $\KK$-subalgebra. Then $A$ is conjugate to a maximal subalgebra of one of the three types S1, S2, or S3. More precisely, exactly one of the following holds: 
\begin{enumerate} 
\item[(i)] $A$ is not semisimple and $A$ is conjugate to $B(k,n-k)$, for some $0 < k <n$ (so $A$ is of type ``S1"). 
\item[(ii)] $A$ is a simple subalgebra of $M_n(D)$, say $A\cong M_t(\Delta )$, for $\Delta$ a division ring with $\KK \subset Z(\Delta)$, and either:
\begin{itemize}
\item[(a)] The bicommutant of $A$ is $C^2(A)=A$, and $A$ is a maximal $Z(D)$-subalgebra of the central simple $Z(D)$-algebra $M_n(D)$; in this case, $A=C(F)$, where $F$ is a minimal field extension of $Z(D)$ contained in $M_n(D)$ (the algebra is of type ``S2"); or
\item[(b)] The bicommutant of $A$ is $C^2(A)=M_n(D)$, $Z(\Delta ) \subset Z(D)$ is a minimal field extension and $Z(D)\otimes Z(\Delta) A\cong Z(D)A=M_n(D)$, and the extension $A\subseteq M_n(D)$ is (canonically) isomorphic to the extension of algebras $Z(\Delta)\otimes_{Z(\Delta)} A\subseteq Z(D)\otimes Z(\Delta) A\cong Z(D)A$. Hence, the algebra is of type ``S3" in this case. 
\end{itemize}
\end{enumerate}
\end{lemma}   

\begin{proof} 
If $A$ is not semisimple, then $J = J(A) \neq 0$. Since the simple left $M_n(D)$-module $M = D^n$ is $A$-faithful, $JM \neq 0$. But note that $JM$ is invariant under right multiplication by $D$, where we consider the usual right $D$-module structure of $M_n(D)$. Hence, $0 \subset JM \subset M$ is a filtration of right $D$-vector spaces and left $A$-modules. Choose a $D$-basis for $M$ compatible with this filtration, and let $X \in M_n(D)$ be the corresponding change-of-basis matrix. Then conjugation by $X$ is an algebra automorphism of $M_n(D)$ which carries $A$ into a subalgebra $A'$ of the subalgebra $B(k,n-k)$ of block upper-triangular matrices, where $k = \dim(JM{}_D)$. By maximality of $A$ and $A'$, $A' = B(k,n-k)$ and the claim follows.   

Assume now $A$ is semisimple. Let $M$ be as before. Then we can write $M$ as $\bigoplus_{i=1}^s{M_i}$, where we let $M_i$ be the direct sum of all submodules of $M$ isomorphic to the simple corresponding to the $i^{th}$ block in the Wedderburn decomposition of $A$. Note that for all $d \in D$, right multiplication by $d$ is an $A$-module endomorphism of $M$; in other words, it is an element of $\Hom_A(M,M) = \bigoplus_{i,j}{\Hom_A(M_i,M_j)} = \bigoplus_i{\End_A(M_i)}$, where the last equality follows by Schur's Lemma. In other words, right multiplication by $d$ restricts to an $A$-module endomorphism of $M_i$, for each $i$, and so each $M_i$ is a right $D$-vector space. Choosing a $D$-basis for $M$ which respects the direct sum decomposition $\bigoplus_i{M_i}$, we see that $A$ is conjugate to a block diagonal matrix subalgebra of $M_n(D)$; obviously, this is maximal only if $s = 1$. Since $M$ is $A$-faithful, this implies that $A$ has only one type of simple module up to isomorphism, and is therefore a simple algebra. 

To prove the final two claims, we first note that $C^2(A)$ is a $Z(D)$-subalgebra and $A \subset C^2(A) \subset M_n(D)$. If $A = C^2(A)$, then $A$ a $Z(D)$-subalgebra of $M_n(D)$. Since it is maximal as a $\KK$-subalgebra, it is also clearly maximal as a $Z(D)$-subalgebra, and the claim follows from \cite{Rac0}.

If $C^2(A) = M_n(D)$, then $C(A)=C^3(A) = C(M_n(D)) = Z(D)$, and hence $Z(\Delta ) = Z(A) = A \cap C(A) \subset C(A) = Z(D)$ (here, we abuse notation slightly by omitting the natural embedding maps $\Delta\hookrightarrow A=M_t(\Delta)$ and $D\hookrightarrow M_n(D)$; hence, we obtain an embedding $Z(\Delta)\hookrightarrow Z(D)$). This inclusion is strict, since otherwise $Z(D)$ would be contained in $A$ and we could reduce to the previous case. We show that this is a minimal field extension. Indeed, if $F$ is a field such that $Z(\Delta ) \subsetneq F \subsetneq Z(D)$, then we have a surjective map $F\otimes_{Z(\Delta )}A \rightarrow FA$ defined by $f\otimes a \mapsto fa \in M_n(D)$. But $A$ is central simple over its center, and $F$ is simple, so that $F\otimes_{Z(\Delta)}A$ is simple, and hence $FA\cong F\otimes_{Z(\Delta )}A\cong M_t(F)$. By $\KK$-dimension or $Z(\Delta)$-dimension, this algebra would then be strictly contained between $A$ and $M_n(D)$. By the same argument, it also follows that $Z(D)\otimes_{Z(\Delta)} A\cong M_n(D)$, and the last part is now automatic. 
\end{proof}  

\noindent Maximal subalgebras of central separable algebras over arbitrary commutative rings were studied in \cite{Rac}. 

\subsection{The Semisimple Case}

To get the complete picture, we now extend the results and describe maximal subalgebras of semisimple algebras.

\begin{definition} 
For a ring $R$ and natural numbers $n, k$, we set $\Delta^k(n,R)$ to be the image of the diagonal map $M_n(R) \rightarrow \prod_{i=1}^k{M_n(R)}$ taking $X \mapsto \prod_{i=1}^k{X}$. We call this the {\bf{diagonal ring}} corresponding to $\prod_{i=1}^k{M_n(R)}$. If $n$ and/or $R$ are understood, we will write $\Delta^k(n,R) = \Delta^k(R) = \Delta^k(n)$.
\end{definition}

\begin{lemma} \label{l.3.2}
Let $B$ and $C$ be (not necessarily finite-dimensional) $\KK$-algebras, with $A \subset B\times C$ a maximal subalgebra of $B\times C$. If the restriction of $\pi_B : B\times C \rightarrow B$ to $A$ is not surjective, then $\pi_C$ is surjective, $\pi_B(A)$ is a maximal subalgebra of $B$, and $A=\pi_B(A)\times C$.
\end{lemma} 
\begin{proof} 
Consider the inclusions $A \subseteq \pi_B(A)\times \pi_C(A) \subseteq \pi_B(A)\times C \subseteq B\times C$. Since $\pi_B(A) \neq B$, $\pi_B(A)\times C$ is a proper subalgebra of $B\times C$ containing $A$. By maximality, $A = \pi_B(A)\times \pi_C(A) = \pi_B(A)\times C$, which implies that $\pi_C$ is surjective. If $X$ is a $\KK$-algebra with $\pi_B(A) \subset X \subset B$, then the containment $A = \pi_B(A)\times C \subset X\times C \subset B\times C$ implies $\pi_B(A) = X$ or $X = B$, and the claim follows. 
\end{proof}   

\begin{proposition} 
Let $D$ be a division ring, with $\KK \subset D$ a field contained in the center of $D$. Suppose that $A \subset M_n(D)\times M_n(D)$ is a maximal $\KK$-subalgebra, such that each projection is surjective when restricted to $A$. Then there is a $\KK$-algebra automorphism $\alpha : M_n(D)\rightarrow M_n(D)$ such that $A = [\id_{M_n(D)}\times \alpha ](\Delta^2(n,D) )$. If $\KK = Z(D)$ and $\dim_{\KK}D<\infty$, then $A$ is conjugate to $\Delta^2(n,D)$ under an inner automorphism.
\end{proposition} 
\begin{proof} 
We first show that $\Delta^2(n,D)$ is a maximal subalgebra of $M_n(D )\times M_n(D)$. If there is algebra $\Delta^2(n,D) \subsetneq A \subsetneq M_n(D )\times M_n(D)$, then $A$ contains an element $(X, Y)$, with $X\neq Y$. Since $\Delta^2(n,D) \subset A$, then $A$ contains the element $(H,0)$ with $H=X-Y\neq 0$. It follows that $(Z,Z)\cdot (H,0)\in A$ and $(H,0)\cdot(Z,Z)\in A$ for all $Z\in M_n(D)$, so $M_n(D)HM_n(D)\times \{0\} =M_n(D)\times \{0\}\subseteq A$. Similarly $\{0\}\times M_n(D) \subset A$; but then $A = M_n(D )\times M_n(D )$, a contradiction. So $\Delta^2(n,D)$ is maximal in $M_n(D)\times M_n(D)$. \\
By hypothesis, for each $X \in M_n(D)$, there is $\alpha (X)$ such that $(X, \alpha (X)) \in A$. Note that such $\alpha(X)$ is unique:  $(X,Z) \in A$ with $Z\neq X$, then $(0,\alpha(X)-Z)\in A$ with $H=\alpha(X)-Z\neq 0$, and we may repeat the argument above: using surjectivity of the projections, for every $Z_1,Z_2\in M_n(D)$, there are elements $(Z_1',Z_1), (Z_2',Z_2)\in A$ so $(0,Z_1HZ_2)=(Z_1',Z_1)\cdot(0,H)\cdot(Z_2',Z_2)\in A$ implies $\{0\}\times M_n(D)\subseteq A$, and similarly $M_n(D)\times \{0\}\subseteq A$, leading to a contradiction. The uniqueness of $\alpha(X)$ now easily implies that $\alpha$ is linear, and in fact an algebra endomorphism (since $(X,\alpha (X))\cdot (Y, \alpha (Y)) = (XY,\alpha(X)\alpha(Y))$). Moreover, $\alpha$ is injective since otherwise some element of the form $(X,0)$ would belong to $A$; surjectivity follows (or can be deduced applying the surjectivity of the first projection). Hence $\alpha$ is an automorphism and the final claim follows by Skolem-Noether. 
\end{proof}

\begin{theorem}\label{t.3.1}  
Let $A \subset B = \prod_{i=1}^t{B_i}$ be a maximal subring (resp. $\KK$-subalgebra) of $B$, where $B_i = M_{n_i}(D_i)$ for some division ring (resp. $\KK$-algebra). Furthermore, suppose that the restriction to $A$ of the projection map $\pi_i : B \rightarrow B_i$ is surjective, for each $i\le t$. Then there exists a pair $i \neq j$ such that $B_i = B_j$, and there is a ring (resp. $\KK$-algebra) automorphism $\Phi$ of $B$, such that $\Phi(A)$ is the direct product of the maximal diagonal subring (resp. $\KK$-subalgebra) $\Delta^2(B_i)\subset B_i\times B_i=B_i\times B_j$ with the ring (resp. $\KK$-algebra) $\ds \prod_{k\not\in \{ i,j \} }{B_k}$. Furthermore, if  all $B_i$ are central $\KK$-algebras, then $\Phi$ can be chosen inner. 
\end{theorem} 

\begin{proof} 
Let $S_i$ be the (unique up to isomorphism) simple $B_i$-module. Then $S_i$ is a $B$-module in a natural way, and so an $A$-module by restriction via $A \subset B$. Let $S \le S_i$ be any non-zero $A$-submodule of $S_i$. Since $\pi_i(A) = B_i$, for any $x \in B_i$ we can find an $a \in A$ such that $\pi_i(a) = x$. Therefore, $xS = \pi_i(a)S = aS \subset S$, and so $S$ is a non-zero $B_i$-submodule of $S_i$, and hence $S = S_i$. This shows that $S_i$ is a simple $A$-module. But the inclusion $A \hookrightarrow B$ is an embedding of $A$-modules. Since $\ds B = \bigoplus_{i}{S_i^{n_i}}$ as a $B$-modules and $A$-modules, and $S_i$ are simple $A$-modules, it is a semisimple $A$-module, and hence $A$ is semisimple too as a left $A$-module. 

Note that since $A \subset B$, the $A$-modules $S_1, \ldots , S_t$ exhaust all simple $A$-module (up to isomorphism). We claim that amongst the $S_1$, \ldots , $S_t$, there are exactly $t-1$ isomorphism types of simple $A$-modules. Let $m$ be the number of distinct isomorphism classes of $A$-modules amongst the $S_1$, $\ldots , S_t$. To begin, we claim $m < t$. Indeed, if the $S_i$'s were all non-isomorphic $A$-modules, then the Wedderburn decomposition of $A$ would be $\ds A = \prod_{i=1}^t{M_{d_i}(\tilde{D}_i)}$, where $\tilde{D}_i = \Endo_A(S_i)^{\op}$ and $d_i$ is the number of times $S_i$ appears in the decomposition of $A$ as a module over itself. We claim that $\tilde{D}_i = D_i$ and $d_i = n_i$. It is clear that $D_i^{\op} = \Endo_{B}(S_i)= \Endo_{B_i}(S_i) \subset \Endo_{A}(S_i) = \tilde{D}_i^{\op}$. To prove the reverse inclusion, let $\varphi \in D_i^{\op} = \Endo_B(S_i) = \Endo_{B_i}(S_i)$. Pick $X \in B_i$ and $\tilde{X} \in A$ such that $\pi_i(\tilde{X}) = X$. Then if $v \in S_i$, $\varphi (Xv) = \varphi (\tilde{X}v) = \tilde{X}\varphi (v) = X\varphi (v)$, so that in fact $\varphi \in \Endo_B(S_i)$, which implies $\tilde{D}_i^{\op} = D_i^{\op}$ and hence $\tilde{D}_i = D_i$. That $d_i = n_i$ follows immediately from the surjectivity of the projections. Hence $\ds A = \prod_{i=1}^t{M_{d_i}(\tilde{D}_i)} = \prod_{i=1}^t{M_{n_i}(D_i)} = B$, contradicting the properness of $A$. From this, it follows that $m< t$. 

To see $m \geq t-1$, find a partition $(t_1,\ldots , t_m)$ of $t$ with $t_1 \geq t_2 \geq \ldots \geq t_m$ such that, after possibly permuting the matrix factors $B_i$, the modules $S_1,\ldots , S_{t_1}$ represent the first isomorphism class of $A$-modules, $S_{t_1+1},\ldots , S_{t_1+t_2}$ represent a distinct isomorphism class, etc. By the above argument, $t_1 \geq 2$. We claim that $t_1 = 2$. Indeed, since $S_1 \cong S_j$ for $j \le t_1$, we also have $D_1 = \Endo_A(S_1)^{\op} \cong \Endo_A(S_j)^{\op} = D_j$ for $j \le t_1$. Using the surjectivity of the projections, $n_1 = n_j$ for $j \le t_1$ as well. 

Consider a decomposition isomorphism $A=A'\oplus I$, where $A'$ is the block corresponding to the simple isomorphic $A$-modules $S_j$ with $j\leq t_1$. The isomorphim of left $A$-modules $S_1\cong S_j$ for $j\leq t_1$ implies $\ker({\pi_1}{\vert_{A}})=\ker({\pi_j}{\vert_{A}})=I$ and the surjectivity of these projections then shows that the maps $\Phi_j=\pi_j\vert_{A'}:A'\rightarrow M_{n_j}(D_j)=M_{n_1}(D_1)$ are bijective. Moreover, considering $A'$ as a block of $A$ with its algebra (ring) strucutre, the maps $\Phi_j$ become isomorphisms of algebras; they are also unital, since if $1=e+f\in A'+I$ ($e\in A', f\in I$), then $\pi_j(f)=0$ so $\pi_j(e)=\pi_j(1)=I\in M_{n_1}(D_1)$. Thus, applying the automorphism $\prod\limits_{j=1}^{t_1}(\Phi_1\circ \Phi_j^{-1})\times \prod\limits_{j>t_1}{\rm Id}$ of the algebra $B=\prod\limits_{j=1}^t B_j$, we see that the algebra $A$ is sent to the algebra $\ds \Delta^{t_1}(n_1,D_1)\times \prod_{j > t_1}{B_j}$. The latter algebra is maximal if and only if $t_1 = 2$, which implies the claim. By definition, we then have $t_i \le 2$ for all $i \le m$. But by a similar argument, if $t_2 = 2$ then we can find an automorphism of $B$ which carries $A$ into the subalgebra $\ds \Delta^{2}(n_1,D_1)\times \Delta^2(n_3,D_3) \times \prod_{j > t_2}{B_j}$, which is not maximal in $B$. Hence $t_2 = 1$, which implies $t_j = 1$ for all $j \geq 2$. 

The last statement is again a consequence of Skolem-Noether.
\end{proof}  

\begin{remark} \label{r.ss}
Suppose that $B$ is a semisimple $\KK$-algebra whose simples are Schur, with Wedderburn decomposition $B \cong \prod_{i=1}^t{M_{n_i}(\KK )}$. If $A$ is a maximal subalgebra of $B$ and each projection map is surjective, then there exist indices $i\neq j$ such that $n_i = n_j$ and $A$ is conjugate to $\Delta^2(n_i, \KK) \times \prod_{k \not\in \{ i,j \}}{M_{n_k}(\KK)}$. Otherwise $A$ is, up to conjugation (up to an inner automorphism), the product of a maximal subalgebra of $M_{n_i}(\KK)$ with the other matrix factors of $B$. Setting $n = n_i$, we see from the simple classification that since $A$ is necessarily a $Z(\KK ) = \KK$-subalgebra of $M_n(\KK)$, $A$ is either of type S1 or S2. If $\KK$ is also an algebraically closed field, then only subalgebras of type S1 are possible.
\end{remark}

We summarize the results to describe maximal subalgebras of semisimple type in arbitrary finite dimensional algebras.

\begin{corollary}\label{c.ss}
Maximal subalgebras $A$ of semisimple type of a finite dimensional $\KK$-algebra $B$ are in 1-1 correspondence with subalgebras $A'$ of $B/J(B)$, which in turn are described by Remark \ref{r.ss}. If, moreover, $B/J(B)$ is separable, and $A_0$ is a subalgebra of $B$ with $B=A_0\oplus J$ and $A_0=\prod\limits_{i}M_{n_i}(\Delta_i)$, then any maximal subalgebra $A$ of $B$ of semisimple type is conjugate to one of the form $A'\oplus J$ where either $A'=\Delta^2(n_i, \KK) \times \prod_{k \not\in \{ i,j \}}{M_{n_k}(\KK)}$ for some $i,j$ with $n_i=n_j$, or $A'=T \times \prod_{k\neq i}M_{n_k}(\Delta_k)$, where $T$ is a maximal subalgebra of $M_{n_i}(\Delta_i)$ of type S1, S2 or S3. 
\end{corollary}
\begin{proof}
This follows from Lemma \ref{l.3.1}, Theorem \ref{t.3.1} and Remark \ref{r.ss}, and when $B/J(B)$ is separable, Theorem \ref{t.gen} is used.
\end{proof}

\noindent We end this section with a result showing that the maximal algebras of semisimple type produce separable extensions in many interesting cases, thus justifying the alternate name of maximal algebras of ``separable type".

\begin{proposition}
Let $B$ be a finite dimensional algebras such that $B/J(B)$ is separable. Then if $A$ is any subalgebra $A$ of semisimple type, then the extension $A\subset B$ is separable.
\end{proposition}
\begin{proof}
Let $A_0$ be a subalgebra of $B$ with $B=A_0\oplus J(B)$. Since $A$ is of semisimple type, $A\supset J(B)$. Also, by hypothesis $A_0\cong B/J$ is separable, and let $E=\sum\limits_i e_i\otimes f_i\in A_0\otimes_\KK A_0$ be a separability idempotent of $A_0$. Let $E'=\sum\limits_i e_i\otimes_A f_i\in B\otimes_A B$ be the image of $E$ through the canonical map $\phi:A_0\otimes_\KK A_0\rightarrow B\otimes_\KK B\rightarrow B\otimes_A B$. We show that $E'$ is a separability idempotent for $B$. First, obviously $\sum\limits_i e_if_i=1\in A_0\subseteq B$. We need to show that $bE'=E'b$ for $b\in B$. Since $B=A_0\oplus J$, it is enough to show this for $b\in A_0$ and for $b\in J$. First, if $b\in A_0$, then 
$$bE'=\sum\limits_i (be_i)\otimes_A f_i=\phi(\sum\limits_i be_i\otimes f_i)=\phi(\sum\limits_i e_i\otimes f_ib)=\sum\limits_i e_i\otimes f_ib=E'b.$$ 
Also, if $b\in J$, then $be_i\in J\subset A$ and $f_ib\in J\subset A$, and using also that $\sum\limits_i e_if_i=1$, we obtain
\begin{eqnarray*}
bE' & = & b(\sum\limits_i e_i\otimes_A f_i) = \sum\limits_i \underbrace{be_i}_{\in  A}\otimes_A f_i = \sum\limits_i 1\otimes_A be_if_i = 1\otimes_A b(\sum\limits_i e_if_i) \\
& = & 1\otimes_A b = b\otimes_A 1 = \sum\limits_i e_if_i b\otimes_A 1 = \sum\limits_i e_i\otimes_A f_ib \,\,\,\,({\rm since\,}f_ib\in A)   \\
& = & E'b
\end{eqnarray*}
which ends the proof.
\end{proof}

\section{Applications and Examples}    \label{s.ae} 

\noindent We now present a series of examples and applications to illustrate the results from the previous sections. We also determine the maximal subalgebras of several important classes of algebras, such as path algebras of quivers and incidence algebras of posets. 

\subsection{Pointed Algebras}
We begin by stating a particular case of the above results, for algebras which are basic Schurian, that is, each simple module is 1-dimensional (such algebras are also called pointed). 

\begin{theorem}\label{t.basic}
Let $B$ be a basic Schurian (i.e. pointed) algebra. Let $e_1,\dots,e_n$ be a complete system of primitive orthogonal idempotents of $B$. Then any maximal subalgebra of $B$ is conjugate to one of the following two types (as before, $J=J(B)$ denotes the Jacobson radical of $B$):\\
(a) ${\rm Span}_\KK(e_i+e_j)+\sum\limits_{k\neq i,j}\KK e_k \oplus J$. \\
(b) $\sum\limits_{k}\KK e_k \oplus H$, where $J^2\subseteq H\subseteq J$ is such that $H/J^2$ is a maximal $B_0$-sub-bimodule of $J/J^2$, where $B_0=\sum\limits_i\KK e_i$. 
\end{theorem}

We now apply this theorem to cases of particular interest, such as path algebras of quivers and incidence algebras. These will follow directly from the previous theorem. 

First we consider quiver algebras; in fact, we reformulate the previous theorem in the language of quivers with relations, when one needs to work with an algebra which is given by a presentation. 

Let $Q$ be a finite quiver. For each pair of vertices $a,b\in Q_0$, we write $Q_1(a,b)$ to denote the set of arrows from $a$ to $b$ and $V(a,b)=\KK Q_1(a,b)$ to denote the span of all arrows $x:a\rightarrow b$ in $Q_1$ (``generalized arrows" from $a$ to $b$). It is well known that every finite dimensional basic Schurian (i.e. pointed) algebra $B$ can be presented as $B=\KK Q/I$ for a finite quiver $Q$ and $I$ an admissible ideal of $\KK Q$, that is $J^n\subseteq I\subseteq J^2$ for some $n$ \cite{Ass5} (in particular, every basic algebra over an algebraically closed field $\KK$). In this case, since the map $\KK Q_0\oplus\KK Q_1\rightarrow \KK Q/I$ is injective, the spaces $V(a,b)$ can be regarded as subspaces of $B=\KK Q/I$. For each pair of vertices $a,b\in Q_0$ and each subspace $V\subseteq V(a,b)$ of codimension $1$ (in particular $V(a,b) \neq 0$), we associate the subalgebra $A(a,b,V)$ of $B= \KK Q/I$ which is defined by $A(a,b,V)=\bigoplus\limits_{a\in Q_0}\KK a \oplus V\oplus \bigoplus\limits_{c,d\in Q_0}V(c,d) \oplus J^2$. This is the subalgebra of $\KK Q/I$ spanned by the images of all paths of length $\geq 2$, all vertices, all arrows $c\rightarrow d$ for $(c,d)\neq (a,b)$ and $V$. Note that $H=V\oplus \bigoplus\limits_{c,d\in Q_0}V(c,d) \oplus J^2$ is a maximal $\KK Q_0$-sub-bimodule of $J$ containing $J^2$. Furthermore, for $a,b\in Q_0$, let $A'(a,b)$ denote the subalgebra ${\rm Span}_\KK\left(\{a+b\}\cup Q_0\setminus \{a,b\}\right)\oplus J$. Then Theorem \ref{t.basic} can be reformulated as follows.

\begin{proposition}
The maximal subalgebras of $B= \KK Q/I$, for a finite quiver $Q$ and admissible ideal $I$, are precisely the algebras conjugate to either $A(a,b,V)$ or $A'(a,b)$. 
\end{proposition}

We note that the extension $A'(a,b)\subseteq \KK Q/I$ is always a separable extension, with $E =$ 
$\sum_{v \in Q_0}{v\otimes_{A'(a,b)}v}$ as a separability idempotent.

In the case of the incidence algebra of a poset $P$, this Theorem also directly produces the structure of all maximal subalgebras. Let $P$ be a finite poset. Recall that the incidence algebra $I(P)$ of $P$ has a basis consisting of pairs $[a,b]$ for $a\leq b$ (the intervals) and multiplication given by ``convolution" $[a,b]*[c,d]=\delta_{b,c}[a,d]$. The following proposition is now a direct consequence of the previous one, or of the Theorem \ref{t.basic}.

\begin{proposition}
Let $A$ be a maximal subalgebra of $I(P)$, where $P$ is a finite poset. Then $A$ is conjugate either to an algebra of the form $I_s(P,a,b):={\rm Span}\{[a,a]+[b,b]\}\oplus\bigoplus\limits_{[c,d]\neq [a,a], [b,b]; \, c\leq d} \KK [c,d]$ with $a\neq b$, $a,b\in P$ (maximal subalgebras of semisimple type) or to one of the form $I_t(P,a,b):=\bigoplus\limits_{[c,d]\neq [a,b]}\KK [c,d]$ where $a,b\in P$, and $a<b$ is a covering relation, that is $a<b$ is minimal (if $a\leq x\leq b$ then either $a=x$ or $b=x$). 
\end{proposition}

We note that incidence algebras of quasi-ordered sets are also considered in literature (and are sometimes called Structural Matrix Algebras; a quasi-ordered set is a set with a partial order $\leq$ which is reflexive and transitive, but not necessarily symmetric). Up to Morita equivalence, any incidence algebra of a quasi-ordered $P$ set is equivalent the incidence algebra of the poset $P/\sim$, where $\sim$ is the equivalence relation generated by symmetrization ($x\sim y$ if $x\leq y$ and $y\leq x)$. Such incidence algebras will appear implicitly in what follows.

\subsection*{Dynkin quivers and other examples}

We will consider a slight relaxation of the notion of a maximal subalgebra $A$ of $B$, which will significantly expand the class of examples and applications. Of course, generically when $A$ and $B$ are not fixed, one can can equivalently talk about minimal extensions of algebras $A\subseteq B$. Such extensions can be regarded via the associated restriction functor ${\rm Res}^B_A:B{\rm-Mod}\rightarrow A{\rm-Mod}$ between the corresponding module categories; hence, it will be useful from a representation theoretic point of view to consider such extensions ``up to Morita equivalence". We thus introduce the following definition.

\begin{definition}
Let $A,B$ be finite dimensional $\KK$-algebras, and $F:B{\rm-Mod}\rightarrow A{\rm-Mod}$ a functor. We say that $F$ is a (minimal) restriction if there is a (minimal) extension of algebras $A'\subset B'$ with $A',B'$ Morita equivalent to $A,B$ respectively, and the diagram
$$\xymatrix{ B{\rm-Mod} \ar[r]^F\ar[d]_{\approx} & A{\rm-Mod}\ar[d]^{\approx} \\ B'{\rm-Mod} \ar[r]_{{\rm Res}^{B'}_{A'}} & A'{\rm-Mod}}$$
is commutative, with vertical arrows being equivalences. If, given the algebras $A,B$, such a (minimal) restriction $F$ exists, we say $A,B$ is a (minimally) embeddable pair. We write $A\ll B$ when $(A,B)$ is an embeddable pair.
\end{definition}

\subsection{Embeddings of quivers of ``separable type"} 
First note that if $Q$ is a Dynkin quiver, (or more generally an acyclic quiver) then one can introduce a partial order on the vertex set $Q_0$ given by paths: $a\leq b$ if there is a path from $a$ to $b$. Without loss of generality, we may assume $Q_0=\{1,2,\dots,n\}$. In the Dynkin case, or more generally, if the underlying graph of $Q$ is a tree, the path algebra is isomorphic to the incidence algebra associated to $(Q_0,\leq)$. This isomorphism $\varphi$ takes a path $p$ between vertices $i$ and $j$ to the matrix element $e_{ij}\in M_n(\KK)$. Here, we regard $\varphi$ also as a morphism $\varphi: \KK Q\rightarrow M_n(\KK)$, and $\varphi(\KK Q)$ is the structural matrix algebra (incidence algebra); of course, as a representation, $\varphi$ yields an indecomposable representation of $Q$; when $Q$ is of type ${\mathbb A}_n$, this is the indecomposable of largest dimension. 

Let $Q$ be Dynkin, or more generally, a tree. We use the above $\varphi$ to create a minimal embedding. Let $A=\varphi(\KK Q)$ and consider two adjacent vertices $i,j$, so that there is an arrow $i\rightarrow j$ in $Q$, and let $\leq$ be the order as above. We consider a new (quasi-)ordering $\preceq$ on $Q_0$ generated (via transitivity) by ``introducing" the new relation $j\preceq i$. Hence, the new relation $\preceq$ is defined by $k\preceq l$ if either $k\leq l$ or $k\leq j$ and $i\leq l$. The incidence algebra $B$ associated to $\preceq$, as a subalgebra of $M_n(\KK)$ is the algebra with basis $\{e_{kl}|k\preceq l\}$, which is exactly the subalgebra generated by $A$ and $e_{ji}$. Let $I$ be the ideal of $B$ generated by $e_{ji}$; it has a basis the elements $e_{kl}$ for which $k\leq j$ and $i\leq l$. Obviously, $B=A+ I$. 
In certain cases, this embedding becomes a minimal embedding.

\begin{proposition}
Let $Q$ be a quiver whose underlying graph is a tree, and $a:i\rightarrow j$ an arrow in $Q$ such that the following condition (*) is satisfied:
\begin{center}
(*) $i$ emits no other arrows except $a$, and $j$ receives no other arrows but $a$. 
\end{center}
Then the algebra extension $A\subseteq B$ is a minimal extension, with $A$ a maximal subalgebra of $B$ of semisimple type. 
\end{proposition} 
\begin{proof}
The condition in the hypothesis shows that $B=A+\KK e_{ji}$, so $A$ is maximal. The subspace $M$ of $B$ spanned by $\{e_{ii},e_{jj},e_{ij},e_{ji}\}$ has as a complement the space $L$ spanned by the $e_{kl}$'s with $k\leq l$ and $\{k,l\}\not\subset \{i,j\}$; this $L$ is an ideal. Hence, $M_2(\KK)$ is necessarily a block of $B$, but it is not a block of $A$. Thus, by the results of Section \ref{s.gen} the maximal subalgebra $A$ of $B$ is of semisimple type.
\end{proof}


We note now that in general, even in the absence of condition (*), the algebra $B$ above is Morita equivalent to the path algebra of $Q'$, where $Q'$ is the quiver obtained from $Q$ by {\it collapsing} the edge $i\rightarrow j$; that is, $Q'$ is obtained from $Q$ by removing $i\rightarrow j$ and then identifying vertex $i$ with vertex $j$. Note that in this case, new paths arise in $Q'$ by concatenating paths ending at $j$ and paths starting at $i$; the paths in $Q$ which contain the arrow $i\rightarrow j$ are in one-to-one correspondence with a subset of the paths in $Q'$ which contain the vertex $i=j$. One way to observe this Morita equivalence is by noting that $B/J(B)\cong \KK^{n-1}\times M_2(\KK)$, where the $M_2(\KK)$ block corresponds to the idempotents $e_{ii}, e_{jj}$. Hence, if $e:=\sum\limits_{t\neq j}e_{tt}$, then $eBe$ is Morita equivalent to $B$ (it is the basic algebra associated to $B$); moreover, in $eBe$, a basis is given by $ee_{kl}e$, for $k,l\in \{1,...,n\}\setminus \{j\}$ and $k\preceq l$. These correspond exactly to paths in $Q'$ from some $k$ to some $l$. This last fact shows that there is an isomorphism $eBe\cong \KK Q'$. This last assertion shows that we have the following.

\begin{proposition}
With the notations above, the pair $\KK Q,\,\KK Q'$ is an embeddable pair. If condition (*) is satisfied, then this is a minimally embeddable pair. 
\end{proposition} 

More generally, if $P$ is a finite poset and $I(P)$ its incidence algebra, we say that an interval $[a,b]$ is {\bf{clamped}} if $x \le b$ implies that $x$ is comparable to $a$, and $a \le y$ implies that $y$ is comparable to $b$. We say that $b$ {\bf{covers}} $a$ if $[a,b] = \{ a,b\}$. If $[a,b]$ is a clamped interval with $b$ covering $a$, then adding the relation $b \le a$ yields a quasi-ordered set $Q$, and a minimal extension of algebras $I(P) \subset I(Q)$. $I(Q)$ is Morita equivalent to the incidence algebra of the poset obtained by collapsing the arrow $a \rightarrow b$ in the Hasse diagram of $P$ to a point. If $Q$ is a quiver whose underlying graph is a tree, then it is naturally an incidence algebra, and the condition (*) above is equivalent to the interval $[i,j]$ being clamped and $j$ covering $i$. In fact, using the remarks above and in Section 1, one can see that the resulting functors $\KK Q'{\rm-Mod}\longrightarrow \KK Q{\rm-Mod}$ can be re-interpreted as being exactly localization functors.

The previous considerations allow us to provide many natural examples of embeddable pairs, which produce separable extensions. For example,  ${\mathbb A}_{n+1}$ ``embeds" in ${\mathbb A}_n$ by ``collapsing" one edge; below, the dotted arrow in ${\mathbb A}_{n+1}$ gets collapsed. 

$$\xymatrix{
{\mathbb A}_{n+1}\ll {\mathbb A}_n & \bullet_1 \ar@{-}[r] &  \dots \ar@{-}[r] & \bullet_{k-1}\ar@{--}[r]\ar@{=>}[d] & \bullet_k \ar@{-}[r] & \dots\ar@{-}[r] & \bullet_{n+1} \\
& \bullet_1 \ar@{-}[r] &  \dots \ar@{-}[r] & (\bullet_{k-1}=\bullet_{k})\ar@{-}[r] & \bullet_{k+1} \ar@{-}[r] &  \dots\ar@{-}[r] & \bullet_{n+1} 
}$$
  
For instance, if $k =2$, this is just the embedding $B(1^{n+1}) \subset B(2^11^n)$, where we use the notation of Definition \ref{d.but} and the note below (here R = $\KK$.) Given the appropriate orientations so that the collapsed arrow satisfies condition (*) above, this becomes a minimal embedding. We give a few more examples and note that there are embeddable pairs $\DD_{n+1} \ll \Aff_n$, and $\DD_{n+1}\ll \DD_n$. This is perhaps also interesting as ``pictorially" one perhaps normally expects to embed lower order $\Aff_n$ and $\DD_n$'s into higher ones. We draw only the diagrams, with the arrow to be collapsed drawn as a dotted arrow (thus collapsing an arrow produces an embedding up to Morita equivalence). Again, with appropriate orientations satisfying condition (*) we obtain minimal embeddings (hence, examples of maximal subalgebras).

$$
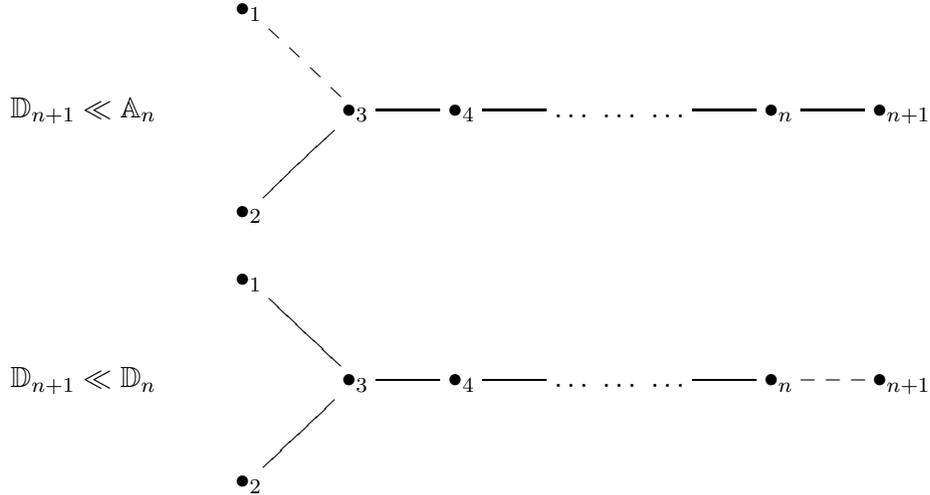
\xymatrix{ & \bullet_1\ar@{--}[dr] & & & & & \\
\DD_{n+1} \ll \Aff_{n} & & \bullet_3\ar@{-}[r] &  \bullet_4\ar@{-}[r] & \dots \,\, \dots \,\, \dots\ar@{-}[r] & \bullet_{n}\ar@{-}[r] & \bullet_{n+1} \\
& \bullet_{2}\ar@{-}[ur] & & & & &}$$ 	
	
$$
\xymatrix{ & \bullet_1\ar@{-}[dr] & & & & & \\
\DD_{n+1} \ll \DD_{n} & & \bullet_3\ar@{-}[r] &  \bullet_4\ar@{-}[r] & \dots \,\, \dots \,\, \dots\ar@{-}[r] & \bullet_{n}\ar@{--}[r] & \bullet_{n+1} \\
& \bullet_{2}\ar@{-}[ur] & & & & &}$$ 	
	
Obviously, in a similar way one can produce many other embeddings, which become minimal embeddings given appropriate orientations on the quiver; these will always be separable. We list a few such possible embeddings which are obtained just as above and leave it to the reader to imagine/draw the appropriate diagrams: 
\begin{center}
$\EE_8 \ll \EE_7$, $\EE_7\ll \EE_6$, $\EE_6\ll \DD_5$, $\EE_8\ll \Aff_7$, $\EE_7\ll \Aff_6$, $\EE_6\ll \Aff_5$. 
\end{center}

\subsection{ Embeddings of quivers of ``split type"}
	
We give another series of examples of embeddings of quivers, which will often produce examples of split extensions. These are embeddings of path algebras that are obtained whenever a subquiver $\Gamma$ of a quiver $Q$ is considered. Let $Q$ be an acyclic quiver, $a\in Q_0$ and let $I$ be the span of all {\it non-trivial} paths in $Q$ which pass through $a$. This is an ideal of $\KK Q$. Now consider $A$ the subalgebra of $\KK Q$ generated by all paths which {\it do not} pass through $a$; this is spanned by all these paths (including the ones of length 0 which are vertices $b\in Q_0\setminus \{a\}$), together with the identity element $1$, and hence contains $a$. Then, we have the following straightforward observation.

\begin{lemma}
With the above notations, $\KK Q=A\oplus I$ is a split extension of $A$, with $I^2=0$.
\end{lemma}
\begin{proof}
The direct sum is obvious as any path either contains $a$ or doesn't. If $p,q$ both contain $a$, then $pq=0$, since otherwise the path $pq$ passes through $a$ twice and $Q$ would contain cycles. 
\end{proof}

This gives again a multitude of examples of embeddings of quivers. The algebra $A$ is itself a path algebra: if $Q_{-a}=Q\setminus \{a\}$ is the quiver obtained from $Q$ by removing all the arrows adjacent to $a$ (but keeping the vertex $a$ itself), then one easily sees that there is a natural identification as $A=\KK Q_{-a}$. Thus, this produces a split extension $\KK Q_{-a} \subset \KK Q$. Again, as before, there are circumstances under which this becomes a maximal embedding. Of course, a more general example of such embeddings is by simply taking $\Gamma=Q_{-S}$ to be the subquiver of $Q$ obtained by removing all the edges whose source or target is a member of $S\subset Q_0$. Then $\KK Q_{-S}\subset \KK Q$ is again a split extension, and $\KK Q_{-S},\KK Q$ an embeddable pair. We note that inside the algebra $\KK Q_{-S}$, the idempotents corresponding to vertices of $S$ are disconnected; one can remedy that by considering the subalgebra $A'$ generated by $\sum\limits_{s\in S}s$ together with all ``other" paths in $Q_{-S}$ (which are ``contained" in $Q\setminus S$), or the subalgebra $A''$ of $A'$ of semisimple type obtained by ``joining" the idempotent $\sum\limits_{s\in S}s$ with some $a\in A'\cap Q_0$. Then both $A'$ and $A''$ are path algebras of corresponding suitable quivers $\Gamma'$ or $\Gamma''$ obtained from $Q$ by {\it erasing} the arrows in the part contained in $S$ and collapsing everything to a point (respectively, erasing that part all together), and these give rise to embeddable pairs $\Gamma',Q$ and $\Gamma'',Q$. 

Suppose that $a \in Q_0$ is a leaf, i.e. its valence in the underlying graph of $Q$ is $1$. Let $b$ be the other vertex showing up in the unique edge adjacent to $a$. Using the notation of section 4.1, consider the minimal extensions $A'(a,b) \subset \KK Q_{-a} \subset \KK Q$. Then $\KK Q_{-a} \subset \KK Q$ is a trivial extension, and hence split. Furthermore, $\KK a$ is an ideal in $\KK Q_{-a}$ and $\KK Q_{-a} = A'(a,b)\oplus \KK a$ as $A'(a,b)$-bimodules. Hence, $\KK Q_{-a}$ is a split extension of $A'(a,b)$. It is easy to note that the quiver of $A'(a,b)$ is obtained from $Q$ by contracting the edge between $a$ and $b$ to a point. Below we list a few such embeddings between Dynkin quivers. In each picture, the vertex to be ``isolated" by the procedure of the above Proposition is depicted by the symbol $\circ$, while the other vertices of the quiver are depicted as full dots $\bullet$. These are usual embeddings that one often considers.

$$
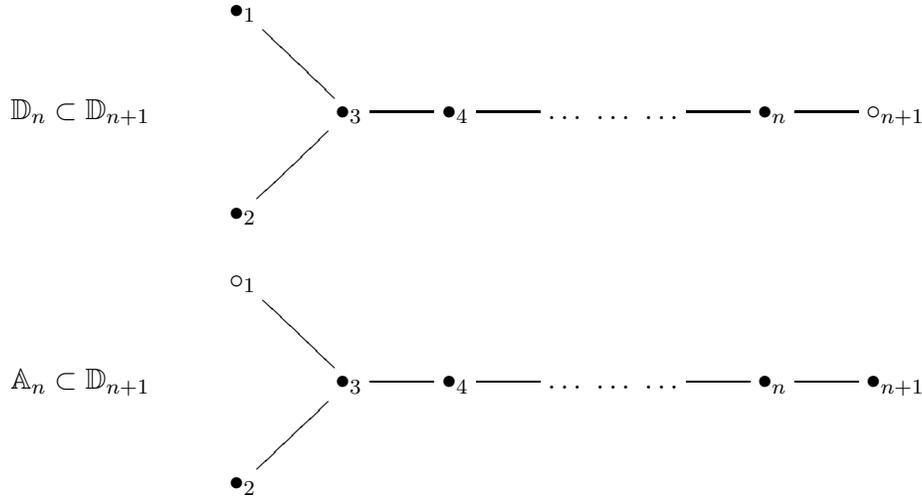
\xymatrix{ & \bullet_1\ar@{-}[dr] & & & & & \\
\DD_{n} \subset \DD_{n+1} & & \bullet_3\ar@{-}[r] &  \bullet_4\ar@{-}[r] & \dots \,\, \dots \,\, \dots\ar@{-}[r] & \bullet_{n}\ar@{-}[r] & \circ_{n+1} \\
& \bullet_{2}\ar@{-}[ur] & & & & &}$$ 	

$$
\xymatrix{ & \circ_1\ar@{-}[dr] & & & & & \\
\Aff_{n} \subset \DD_{n+1} & & \bullet_3\ar@{-}[r] &  \bullet_4\ar@{-}[r] & \dots \,\, \dots \,\, \dots\ar@{-}[r] & \bullet_{n}\ar@{-}[r] & \bullet_{n+1} \\
& \bullet_{2}\ar@{-}[ur] & & & & &}$$ 	

Hence, this case is one where, as ``pictorially" expected, lower order $\Aff_n$ and $\DD_n$'s embed into the ones of higher order.

\subsection{ Remarks on the representation theory} 

 We note that whenever $\varphi:\Gamma\rightarrow Q$ is a map between quivers which is a morphism of partial semigroups between the partial semigroups of paths in $\Gamma$ and $Q$ respectively, then the induced morphism of algebras $\varphi:\KK \Gamma\rightarrow \KK Q$ produces a restriction functor $Res_\varphi: \KK Q{\rm-Mod}\rightarrow \KK \Gamma{\rm-Mod}$ which respects tensor products. This can be observed directly with quiver representations, or using the following fact: the structure as a monoidal category of $\KK Q{\rm-Mod}$ (tensor product) for a quiver $Q$ is naturally associated to the partial semigroup algebra comultiplication $\Delta: \KK Q\rightarrow \KK Q\otimes_\KK \KK Q$, $\Delta(p)=p\otimes p$ for all paths $p$. The above morphism $\varphi$ is compatible with this comultiplication ($\Delta(\varphi(p))=(\varphi\otimes\varphi)\Delta(p)$), and this implies that the functor $Res_\varphi$ commutes with the tensor product of objects (in other words, this is a tensor functor, see \cite{EGNO}). Thus, the previous procedure of deleting arrows between vertices in some subset $S \subset Q_0$, and the resulting embeddings of quivers $Q_{-S}\subset Q$ all give rise to morphisms between the representation rings of the underlying quivers. 

We also note that Lemmas \ref{l.1.1}, \ref{l.1.2} and Corollaries \ref{c.1.1}, \ref{c.1.2} can be used to relate some indecomposables over quivers and their sub-quivers, and algebras and their subalgebras. For quiver of types $\Aff_n$ and $\DD_n$, these can be observed directly using the structure of the indecomposables. Indeed, for example, via the embeddings of the previous subsection, every indecomposable over $\DD_n$ is obtained as the restriction of an indecomposable over $\DD_{n+1}$, and every indecomposable $\Aff_n$-representation is obtained as a restriction of an indecomposable $\Aff_{n+1}$-module. One sees that these Lemmas and Corollaries can be interpreted in terms of positive roots: for example, positive roots of $\Aff_n$ are {\it all} obtained from some positive root of $\Aff_{n+1}$ by deleting one entry; similar statements work for $\DD_{n}$ and $\DD_{n+1}$ and $\Aff_{n}$ and $\DD_{n+1}$. 

We end this subsection with two more examples showing how restriction to subalgebras can give interesting representations, showing the potential of considering subalgebras. The previous examples showed an interpretation on how the thin representations (i.e. representations in which multiplicity of simples in the composition series is at most 1) of lower $\Aff_n$ and $\DD_n$ are obtained from higher analogues. We now show examples on how the non-thin indecomposable modules in type $\DD_n$ can be obtained from indecomposable $\Aff_n$-representations by restriction to subalgebras.

\begin{example} 
Consider the following embedding of $\KK$-algebras, written as subalgebras of $M_5(\KK)$.

\begin{eqnarray*}
\left(\begin{array}{ccccc}
* & 0 & 0 & 0 & 0 \\
* & a & b & 0 & 0 \\
0 & 0 & * & 0 & 0 \\
0 & 0 & b & a & * \\
0 & 0 & 0 & 0 & * 	
\end{array}
\right)={\rm Span}(e_{11},e_{21},e_{22}+e_{44},e_{23}+e_{43},e_{33},e_{45},e_{55})
& \hookrightarrow & 
\left(\begin{array}{ccccc}
* & 0 & 0 & 0 & 0 \\
* & * & * & 0 & 0 \\
0 & 0 & * & 0 & 0 \\
0 & 0 & * & * & * \\
0 & 0 & 0 & 0 & * 	
\end{array}
\right)
\end{eqnarray*}

The second algebra $B$ is an $\Aff_5$ with zig-zag orientation. 
 
The first algebra is isomorphic to a quiver algebra of type $\DD_4$, with $e_{22}+e_{44}$ representing the idempotent corresponding to a sink; $e_{11},e_{33},e_{55}$ correspond to the other three vertices, and $e_{21},e_{23}+e_{43},e_{45}$ represent the three edges going out of $e_{22}+e_{44}$. We can represent this embedding symbolically by the following $\DD_4$ and $\Aff_5$ graphs; one can think of this of this embedding as obtained from a ``gluing operation" on the $\Aff_5$ quiver,  through which the $2\rightarrow 3$ and $4\rightarrow 3$ arrows are being identified (glued), to create a $\DD_4$. In the diagram below, the equal signs are intended to refer to this interpretation. 

{\small 
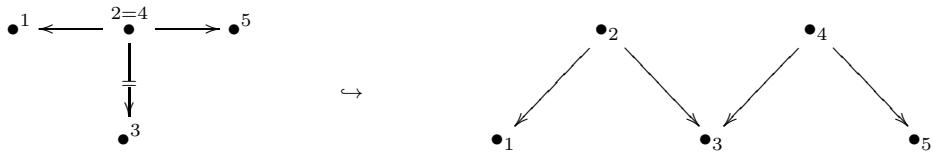
$$\xymatrix{ \bullet^1 & \stackrel{2=4}{\bullet}\ar[l]\ar[r]\ar[d]|{=} & \bullet^5 & \ar@{}[dr]_{\hookrightarrow} &  & & \bullet_2\ar[dl]\ar[dr] &  & \bullet_4\ar[dl]\ar[dr] & \\ & \bullet^{3} & & & & \bullet_1 & & \bullet_3 &  & \bullet_5}$$}

As noted, this results in an algebra embedding. Now, let $M$ be the 5 dimensional representation given by the embedding of $B$ into $M_5(\KK)$ as above (this is also the defining representation of $B$ as an incidence algebra). It is the thin indecomposable dimension of $\Aff_5$ of maximal dimension (5, corresponding to the positive root $(1,1,1,1,1)$ of $\Aff_5$). Restricting this to $\DD_4$, it is not difficult to note that we obtain an indecomposable representation of $A$, which is necessarily the 5-dimensional representation of $\DD_4$ corresponding to the root $(1,2,1,1)$. While this example has fixed an orientation for simplicity, this procedure can be done with any orientation. 

By extending the ``tail" of $\Aff_5$, one can easily generalize this example to an embedding of $\DD_n$ into $\Aff_{n+1}$, which has the result of obtaining the indecomposable $\DD_n$ representation corresponding to the root $\underbrace{(1,2,1,1,...1)}_{n}$ from the indecomposable representation of $\Aff_{n+1}$ corresponding to the root $\underbrace{(1,1,\dots,1)}_{n+1}$ (the one of maximal dimension). We note that at the root level, this amounts to ``joining" two of the entries - second and third one - of the positive root $\underbrace{(1,1,\dots,1)}_{n+1}$ of $\Aff_{n+1}$ to create the root $\underbrace{(1,2,1,\dots,1)}_n$ of $\DD_n$. This construction shows that this formal procedure actually has a representation theoretic meaning (it is a categorification of this procedure on roots).  
\end{example}

\begin{example}
Consider the following embedding of algebras

$$
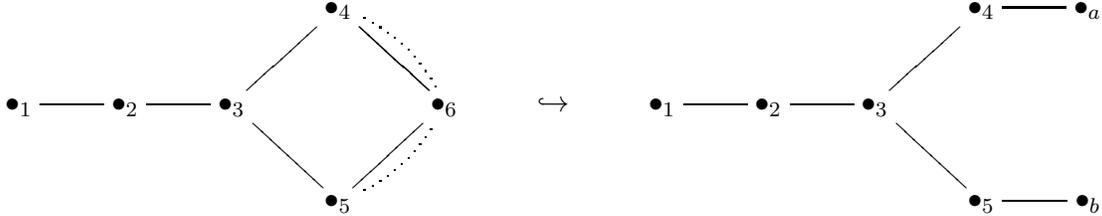
\xymatrix{
 & &  & \bullet_4 \ar@{-}[dr]\ar@{-}[dr]\ar@{..}@/^/[dr] &  &  & &  & & \bullet_4\ar@{-}[r]  & \bullet_{a}  \\
\bullet_1 \ar@{-}[r] & \bullet_2 \ar@{-}[r] & \bullet_3\ar@{-}[ur]\ar@{-}[dr] &  & \bullet_6 & \hookrightarrow & \bullet_1 \ar@{-}[r] & \bullet_2 \ar@{-}[r] & \bullet_3\ar@{-}[ur]\ar@{-}[dr] & &  \\
& & & \bullet_5 \ar@{-}[ur]\ar@{..}@/_/[ur] & & & & & & \bullet_5 \ar@{-}[r] & \bullet_{b} 
}$$

Here, in the second diagram, the algebra is just the path algebra of the quiver (which can have any orientation). In the first part of the diagram, the dotted line means that if the two arrows $\xymatrix{4\ar@{-}[r] & 6}$ and $\xymatrix{6\ar@{-}[r] & 5}$  are oriented such that a path can be formed with them, then there is a $0$ relation in the algebra. The embedding of the two algebras is simply such that $e_6=e_a+e_b$, where $e_i$ denotes the idempotent corresponding to the vertex $i$ in both algebras. This is a maximal subalgebra embedding (i.e. a minimal embedding), and it is of semisimple type, and is so a separable extension. The second one is a quiver algebra (or even incidence algebra) which is not of finite type (it is an Euclidean $\widetilde{\mathbb{E}_6}$), and hence, the first is not of finite type.

We note that this is very close in spirit to covering theory; in fact, we note that if $A$ is an algebra, finding an overalgebra $B$ of $A$ with good properties means that the restriction functor $Res^B_A:B{\rm-Mod}\longrightarrow A{\rm-Mod}$ acts as a ``cover" for $A$-modules. If such an extension is separable, one can exclude finite type of $A$ if such a ``cover" is not of finite type. One can obtain variations of this example by changing the pictures appropriately so that the $B$ is a quiver algebra which is not of finite type.  
\end{example}

\subsection{Maximal subalgebras over non-algebraically closed fields}

\noindent We also give some examples to illustrate maximal subalgebras of semisimple types S2 and S3. 

\begin{example} 
Let $p$ be a prime, $f(X)\in \QQ[X]$ an irreducible polynomial of degree $p$ and $T$ be the companion matrix of $f$. Then the characteristic polynomial of $T$ is $f$ and it coincides with its minimal polynomial. Thus, $F=\KK[T]\subseteq M_p(\QQ)$ is a subalgebra which is a field extension of $\KK$. Hence, its centralizer $C(F)$ in $M_n(\QQ)$ is a maximal subalgebra. We note that, in fact, $C(F)=F=\KK[T]$, which follows by the irreducibility of the minimal polynomial (and is well known in this case). This is a maximal subalgebra of type S2.
\end{example}

\begin{example}
Let $\HH$ be the division algebra of quaternions, with subfields $\RR\subset \CC\subset \HH$. Then $\CC$ is a minimal field extension of $Z(\HH)=\RR$, and so the centralizer $Z(\CC)$ in $M_n(\HH)$ is a minimal extension of $\RR$-algebras. One can check that $Z(\CC)=M_2(\CC)$; indeed, $M_2(\CC)\subset Z(\CC)\subset M_2(\HH)$; now both $Z(\CC)$ and $M_2(\HH)$ are bimodules over $M_2(\CC)$. The quotient $M_2(\HH)/M_2(\CC)$ has dimension $8$, so as an $M_2(\CC)$-bimodule it can only be simple. This shows that there are no intermediate $M_2(\CC)$-bimodules between $M_2(\CC)$ and $M_2(\HH)$, and proves the equality $Z(\CC)=M_2(\CC)$. The algebra $M_2(\CC)$ is a maximal subalgebra of $M_2(\HH)$, which is thus also of type S2.
\end{example}

\noindent Using $\HH$, we give an example of maximal subalgebra of type S3.

\begin{example}
Note that $\RR\subset \CC$ is a minimal field extension; thus the extension $\RR\otimes_\RR M_n(\HH) \hookrightarrow \CC\otimes_\RR M_n(\HH)\cong M_n(M_2(\CC))=M_{2n}(\CC)$ is a maximal $\RR$-subalgebra. One can also construct this, as shown in the discusion on the simple case in Section \ref{s.3}: regard $\HH$ as an $\RR$-subalgebra of $M_2(\CC)$ (it is a maximal subalgebra), and consider the embedding $M_n(\HH)\subset M_{2n}(\CC)$, which gives a maximal $\RR$-subalgebra of $M_{2n}(\CC)$. This is an example of type S3.
\end{example}


\subsection{Maximal dimension of subalgebras}

We note now how our results can be applied to determine the maximal dimension of a subalgebra of $M_n(\KK)$. This problem was considered in \cite{Ag}, over an algebraically closed field. The author did not use the results of \cite{Rac,Rac0} (effectively re-descovering some of these) but instead used a deep result of Gerstenhaber regarding the maximal dimension of a subspace of nilpotent matrices \cite{Ger}. Here we provide a direct argument, based on the above classification; as noted before, in the case of an algebraically closed field the structure of maximal subalgebras is significantly simplified (and does not need the considerations on subalgebras of simple subalgebras of types S2 and S3). In fact, using our approach, we can give a general result for arbitrary finite dimensional algebras; it generalizes the result of \cite{Ag} where it was shown that for $M_n(\KK)$, the maximal dimension of a subalgebra is $n^2-n+1$.

\begin{theorem}
Let $B$ be a finite dimensional algebra over an algebraically closed field $\KK$. Let $J=J(B)$, and $n_1\leq n_2\leq \dots \leq n_t$ be the dimensions of the blocks of $B/J$. Then the maximal dimension of a proper subalgebra of $B$ is  $$\dim(B)-1-\max\{n_1-2,0\}.$$
That is, it is $\dim(B)-n_1+1$ if $n_1>1$ (i.e. if $B$ has no 1-dimensional blocks) and it is $\dim(B)-1$ otherwise.
\end{theorem}
\begin{proof} 
Write $B = B_S \oplus J$, where $B_S\cong B/J$. If $n_1 = 1$, then $B/J$ has a codimension-$1$ subalgebra $B'$, and $A = B'\oplus J$ is a codimension-$1$ subalgebra of $B$. Otherwise $n_1>1$. If $A \subset B$ is a subalgebra of maximal dimension, it is also maximal; if it is of split type, then $J(A)$ is a (maximal) $B_S$-sub-bimodule of $J$ and $\dim_{\KK}(J/J(A)) \geq n_1^2$ ($n_1^2$ is the minimum dimension of a $B_S$-bimodule). Hence, the maximum dimension of such an algebra is $\dim_{\KK}(B) - n_1^2$. If $A$ is of separable type, then $A/J$ is a maximal subalgebra of $B/J$. Corollary \ref{c.ss} shows that the dimension of $A/J$ is either $\dim(B/J)-n_i^2$ (coming from diagonal embeddings) or $\dim(B/J)-k(n_i-k)$ for some $1\leq k<n_i$ (coming from maximal subalgebras of blocks). The largest dimension of such a subalgebra is thus $\dim_{\KK}(B/J) - n_1 +1$, and hence $\dim_{\KK}(A) \le \dim_{\KK}(B) - n_1+1$, and there is always a subalgebra of $B$ which attains this dimension. Since $\dim_{\KK}(B) - n_1 + 1 \geq \dim_{\KK}(B) - n_1^2$ whenever $n_1 > 1$, the result follows. 
\end{proof}

\subsection{Maximal Subalgebras and automorphisms}  

In Section 2 we saw an example with two isomorphic subalgebras of a fixed algebra $B$, which were not isomorphic under any automorphism of $B$. We now examine this behavior in greater detail. If $B$ is a finite-dimensional algebra over a field $\KK$, then $\Aut (B)$ acts of the poset of subalgebras of $B$. In particular, it permutes the collection of maximal subalgebras of $B$. It is natural to ask whether this action determines the isomorphism classes of subalgebras of $B$. In other words, if $A, A' \subset B$ are isomorphic maximal subalgebras, does it follow that there exists an $\alpha \in \Aut (B)$ such that $A' = \alpha (A)$? Unfortunately this does not happen in general, and can fail even in nice enough cases of path algebras of not too complicated quivers, as the following example of an acyclic Schurian quiver illustrates. 

\begin{example}  Let $Q$ denote the following quiver:

 $$
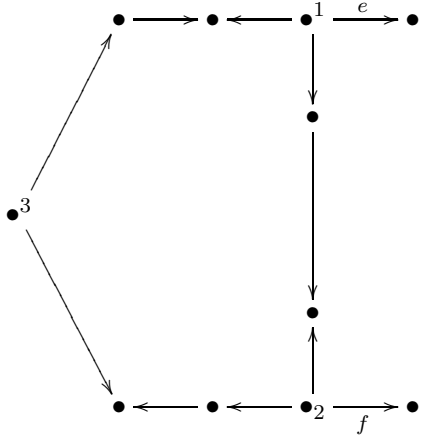
\xymatrix{ & \bullet \ar[r] & \bullet & \ar[l] \bullet^1 \ar[d] \ar[r]^e & \bullet \\ 
& & & \bullet \ar[dd] &  \\  
\bullet^3 \ar[uur] \ar[ddr] & & & &\\  
& & & \bullet & \\
& \bullet & \bullet \ar[l] & \bullet_2 \ar[l] \ar[u] \ar[r]_f & \bullet} $$

Set $B = \KK Q/J^2$, for $\KK$ algebraically closed and $J = J(\KK Q)$ the Jacobson radical of $\KK Q$ (the arrow ideal). Then any $\varphi \in \Aut (B)$ satisfies $\varphi (xJy) = \varphi (x) \varphi (J) \varphi (y) = \varphi (x) J \varphi (y)$, for all $x , y \in Q_0$. Up to inner automorphisms we may assume that $\varphi (Q_0) = Q_0$, and with this assumption $\varphi$ induces a quiver automorphism of $Q$. In particular, no automorphism of $B$ can send vertex $1$ to vertex $2$ (by inspection, $Q$ has no non-trivial automorphisms). Consider the maximal subalgebras $A'(1,3), A'(2,3) \subset B$. Then $A'(1,3)$ and $A'(2,3)$ are isomorphic, but there is no automorphism of $B$ carrying one to the other. Indeed, any proposed automorphism would induce an automorphism of $B$ sending $1$ to $2$, a contradiction. This also happens with algebras of split type. For an edge $\alpha \in Q$, let us denote the maximal subalgebra of split type $A(s(\alpha), t(\alpha) , \Sp_{\KK}Q_1\setminus\{ \alpha\} )$ simply by $A(\alpha )$. Then $A(e) \cong A(f)$ (where $e,f$ are the labeled edges above), but no automorphism of $B$ carries $A(e)$ to $A(f)$.
\end{example} 

Nevertheless, the isomorphism classes of some subalgebras of a finite dimensional algebra $B$ are determined by its automorphism group. For instance, this always holds trivially for $\KK$. More generally, we could fix a subgroup $G \le \Aut (B)$, and ask which isoclasses of subalgebras of $B$ are determined by the action of $G$ on its subalgebra poset?  

\begin{example} Let $Q$ be the Kronecker quiver with two arrows: 
$$
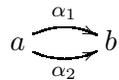
\xymatrix{a \ar@/^/[r]^{\alpha_1} \ar@/_/[r]_{\alpha_2} & b} 
$$

Let $B = \KK Q$. Then up to inner automorphisms there is a unique maximal subalgebra of separable type. By contrast, any one-dimensional subspace $W$ of $\KK Q_1$ yields a maximal subalgebra $A(a,b,W)$ of split type. A simple computation shows that if $W_1$ and $W_2$ are any two such subspaces, then $A(a,b,W_1)$ and $A(a,b,W_2)$ lie in the same $\operatorname{Inn}(B)$-orbit of the subalgebra poset if and only if $W_1 = W_2$. However, $A(a,b,W_1) \cong \KK \mathbb{A}_2 \cong A(a,b,W_2) $, and there is an automorphism of $B$ taking $A(a,b,W_1)$ to $A(a,b,W_2)$. Hence the isomorphism classes of maximal subalgebras of $B$ are not determined by the $\operatorname{Inn}(B)$-orbits, but they are determined by the $\Aut (B)$-orbits. In fact, it is not much harder to check that {\em{any}} isoclass of subalgebras of $B$ is determined by the $\Aut (B)$-orbits.

\end{example}

The collection of subalgebras determined by $G$-orbits is related to the existence of a certain functor between small categories. Let $\mathcal{P}(B)$ denote the poset of $\KK$-subalgebras of $B$. Then the action of $G$ on $\mathcal{P}(B)$ can be considered as a functor $G \rightarrow \operatorname{Cat}$, where $\operatorname{Cat}$ denotes the category of all small categories. We define $\mathcal{P}(B)/G$ to be the colimit of this functor. Its objects are the $G$-orbits of $\mathcal{P}(B)$.

Let $\mathcal{SG}(B)$ denote the subcategory of $\KK\operatorname{-Alg}$ whose objects are the subalgebras of $B$, and whose morphisms are the isomorphisms between such algebras. There is a map on objects $\mathcal{SG}(B) \rightarrow \mathcal{P}(B)/G$ which sends a subalgebra of $B$ to its $G$-orbit. This map does not necessarily extend to a functor. But if its restriction to a full subcategory $\mathcal{C} \subset \mathcal{SG}(B)$ {\em{is}} functorial, then isomorphic objects of $\mathcal{C}$ necessarily lie in the same $G$-orbit of $\mathcal{P}(B)$. Furthermore, a union of such full subcategories is another full subcategory with the same property. Hence, there is a unique largest full subcategory $\mathcal{C}_G(B)$ of $\mathcal{SG}(B)$, such that the natural map on objects $\mathcal{C}_G(B) \rightarrow \mathcal{P}(B)/G$ is a functor. It appears that in general, $C_G(B)$ is difficult to compute. 

\begin{example} 
Let $Q$ be the Kronecker quiver from the previous example, and $B = \KK Q$. Then $C_{\Aut(B)}(B) = \mathcal{SG}(B)$. One can check directly that $C_{\operatorname{Inn}(B)}(B)$ contains four isomorphism classes, represented by $B$, $A'(a,b)$, $\Sp_{\KK}\{ a,b\} \cong \KK \times \KK$, and $\KK$, respectively.
\end{example}

\end{document}